\documentclass[11 pt]{article}
\usepackage{amsfonts}
\usepackage{amsmath}
\usepackage{amsthm}
\usepackage{fullpage}
\usepackage{centernot} 
\usepackage{pb-diagram,enumerate} 

\usepackage{mathrsfs}
 \usepackage{amssymb}
\usepackage[OT2,T1]{fontenc}
\DeclareSymbolFont{cyrletters}{OT2}{wncyr}{m}{n}
\DeclareMathSymbol{\Sha}{\mathalpha}{cyrletters}{"58}

\def\F{{\rm \mathbb{F}}}
\def\Z{{\rm \mathbb{Z}}}
\def\Q{{\rm \mathbb{Q}}}

\def\C{{\rm \mathbb{C}}}
\def\R{{\rm \mathbb{R}}}

\def\P{{\rm \mathbb{P}}}
\def\Fbar{{\rm \bar F}}

\def\J{{\rm \mathcal{J}}}

\def\Aut{{\rm Aut}}

\def\coker{{\rm coker}}

\def\la{{\rm \longrightarrow \,}}

\def\Stab{{\rm Stab}}

\def\avg{{\rm avg}}

\def\ur{{\rm ur}}

\def\Sym{{\rm Sym}}

\def\SL{{\rm SL}}
\def\GL{{\rm GL}}

\def\Vol{{\rm Vol}}
\def\Gal{{\rm Gal}}
\def\Disc{{\rm Disc}}
\def\Delta{{\rm disc}}
\def\Res{{\rm Res}}

\def\Sel{{\rm Sel}}

\def\sol{{\rm sol}}

\def\ls{{\rm loc.\hspace{.5mm} sol.}}

\newtheorem{theorem}{Theorem}
\newtheorem{lemma}[theorem]{Lemma}

\newtheorem{corollary}[theorem]{Corollary}
\newtheorem{proposition}[theorem]{Proposition}
\newtheorem{remark}[theorem]{Remark}

\title{A positive proportion of elliptic curves $y^2=x^3+k$ have rank 0}
\title{Selmer groups and binary cubic forms}
\title{On the 3-isogeny Selmer groups of elliptic curves of \\$j$-invariant zero}
\title{The average size of the 3-isogeny Selmer groups of elliptic curves $y^2 = x^3 + k$}

\author{Manjul Bhargava, Noam Elkies, and Ari Shnidman}

\begin{document}

\maketitle

\vspace{-.1in}
\begin{abstract}
The elliptic curve $E_k \colon y^2 = x^3 + k$ admits a natural 3-isogeny $\phi_k \colon E_k \to E_{-27k}$.  We compute the average size of the $\phi_k$-Selmer group 
as $k$ varies over the integers.  Unlike previous results of Bhargava and Shankar on $n$-Selmer groups of elliptic curves,  we show that this average can be very sensitive to congruence conditions on $k$; this sensitivity can be precisely controlled by the Tamagawa numbers of $E_k$ and $E_{-27k}$.  As consequences, we prove that the average rank of the curves $E_k$, $k\in\Z$,  is less than 1.21 and over $23\%$ (resp.\ $41\%$) of the curves in this family have rank~0 (resp.\ 3-Selmer rank 1).  
\end{abstract}


\section{Introduction}
  
Let $F$ be a field of characteristic not 2 or 3, and let $k \in F$ be non-zero.  The elliptic curve  
\begin{equation}E_k \colon y^2 = x^3  + k\end{equation}
has $j$-invariant 0, and every elliptic curve $E/F$ with $j(E) = 0$ is isomorphic to $E_k$ for some $k \in F$.  The curves $E_k$ and $E_{k'}$ are isomorphic if and only if $k' = km^6$ for some $m \in F^\times$.  

Over a separable closure $\bar F$, each $E_k$ has complex multiplication by the Eisenstein integers $\Z[\zeta_3]$, and hence admits an endomorphism $\sqrt{-3}$ of degree 3.  This endomorphism is not generally defined over $F$, but its kernel is, and consequently there is a 3-isogeny $\phi=\phi_k \colon E_k \to E_{k'}$ for some $k' \in F$.  By duality, there is a $3$-isogeny $\hat\phi=\hat\phi_k\colon E_{k'} \to E_k$ in the reverse direction.  In fact, we may take $k' = -27k$, and then these isogenies are given explicitly as:
\begin{align}
\label{phidef}
\phi_k\colon 
(x,y)& \; \mapsto \,
  \left(\,\frac{x^3+4k}{x^2}\,\,,\,\frac{y(x^3-8k)}{x^3}\right),\\
\hat \phi_k\colon\label{phihatdef}
(x,y) &\; \mapsto \,
  \left(\,\frac{x^3-108k}{9x^2}\,\,,\,\frac{y(x^3+216k)}{27x^3}\right).
\end{align}
After identifying $E_{3^6k}$ and $E_k$, we have $\hat \phi_k = \phi_{-27k}$.

If $F$ is a number field, then the {\it $\phi$-Selmer group} 
\[\Sel_\phi(E_k) \subset H^1(G_F, E_k[\phi])\] associated to the isogeny $\phi$ consists of all ``locally soluble'' cohomology classes, i.e., those that are locally in the image of the connecting map 
\[\partial_v \colon E_{-27k}(F_v) \la H^1(G_{F_v}, E_k[\phi])\]
for every place $v$ of $F$; here, $G_F = \Gal(\bar F/F)$.  \pagebreak


In this paper, we take $F=\Q$ and study the average size of the $\phi$-Selmer group for the elliptic curves $E_k$ as $k$ varies. Let

\begin{equation}
\label{eq:r}
r = \frac{103 \cdot 229}{2 \cdot 3^2 \cdot 7^2 \cdot 13}\prod_{p\,\equiv\, 5 \rm{\:(mod\: 6)}}
\frac{
   (1-p^{-1})
   (1 + p^{-1} + \frac53 p^{-2} + p^{-3} + \frac53 p^{-4} + p^{-5})
   }{1-p^{-6}}. 
\end{equation}
Then we prove:

\begin{theorem}\label{main}
  When the elliptic curves $E_k\colon y^2=x^3+k$, $k\in\Z$, are ordered by the absolute value of~$k$, 
the average size of the $\phi_k$-Selmer group associated to
  the $3$-isogeny $\phi_k \colon E_k \to E_{-27k}$ is $1+r$ if $k$ is negative, and $1+r/3$ if $k$ is positive.
\end{theorem}
We note that, in this theorem, we may take $k$ to range over all
integers or, if desired, only the sixth-power-free ones (so that we obtain each isomorphism class of elliptic curve of $j$-invariant 0 exactly once).
We can calculate the product in~(\ref{eq:r}) efficiently by approximating
it by a product of powers of the values at $s=2,3,4,\ldots$ of
$(1-2^{-s}) (1-3^{-s}) \zeta(s)$ and
$(1+2^{-s}) L(\chi,s)$ where $\chi$ is the Legendre symbol mod~$3$;
we find that the product's numerical value is $1.033735512017364858\ldots$,
so $r = 2.1265\ldots$, making
$1+r= 3.1265\ldots$ and $1+r/3 = 1.7088\ldots$.

In fact, we are able to determine the average size of the $\phi_k$-Selmer group of $E_k\colon y^2=x^3+k$ where $k$ varies in any subset $S$ of $\Z$ defined by finitely many---or, in suitable cases, infinitely many---congruence conditions.  
Let $S \subset \Z$ be any subset of integers defined by sign conditions and congruence conditions modulo a power of each prime, such that for all sufficiently large primes $p$, the closure $S_p$ of $S$ in $\Z_p$ contains all elements of $\Z_p$ not
divisible by $p^2$.  We call such a set {\it acceptable}.  For example, the set of all sixth-power-free
integers is acceptable.  The set of all positive squarefree integers
congruent to 1 (mod~3) is also acceptable.  

For any isogeny $\phi \colon A \to A'$ of abelian varieties over $\Q$ and for any prime $p \leq \infty$, define the {\it local Selmer ratio of $\phi$} at $p$ to be 
\begin{equation}\label{deflsr}
c_p(\phi) := \displaystyle\frac{|A'(\Q_p)/\phi(A(\Q_p))|}{|A[\phi](\Q_p)|}
\end{equation}
where $\Q_p=\R$ when $p=\infty$. 
Then the following theorem gives the average size of the $\phi_k$-Selmer group as $k$ varies over the integers in $S$:

\begin{theorem}\label{mainS}
Let $S$ be any acceptable set of integers.
When the elliptic curves
 $E_k\colon y^2=x^3+k$, $k\in S$, are ordered by the absolute value of $k$, the average size of the $\phi$-Selmer group associated to
 the $3$-isogeny $\phi \colon E_{k}\to E_{-27k}$ is 
\begin{equation}\label{mainSformula}
1 + \prod_{p \leq \infty} \displaystyle\frac{\displaystyle\int_{k \in S_p} c_p(\phi_k) dk}{\displaystyle\int_{k \in S_p} dk},  
\end{equation}
where for $p < \infty$, $S_p$ denotes the $p$-adic closure of $S$ in $\Z_p$ and $dk$ denotes the usual additive Haar measure on $\Z_p$, and for $p=\infty$, $S_p$ denotes the smallest interval in $\R$ containing $S$ and $dk$ denotes the usual additive measure on $\R$. 
\end{theorem}

\begin{remark}
{\em The factor at $p = \infty$ is equal to $\frac13$, $1,$ or $\frac23$, depending on whether $S$ contains only positive, only negative, or both positive and negative integers, respectively. } 
\end{remark}

\begin{remark}
{\em The group $H^1(G_F, E_k[\phi])$ parameterizes isomorphism classes of {\it $\phi$-coverings}, i.e., maps $ f : C \to E_{-27k}$ of  curves that become isomorphic to $\phi$ over $\bar F$.  Therefore, we can interpret the local Selmer ratio $c_p(\phi)$ as the number of soluble $\phi$-coverings $f$ over $\Q_p$, where each $f$ is weighted by the inverse of the order of its automorphism group.  
Theorem \ref{mainS} can thus be interpreted as saying that the expected (weighted) number of nontrivial locally soluble $\phi$-coverings over $\Q$ is simply the product of the (weighted) number of soluble $\phi$-coverings over $\Q_p$.  This suggests a Selmer group analogue of Bhargava's conjectures for number fields \cite{mlocal}, which would include as a special case the conjecture of Bhargava and Shankar
that the average size of $\Sel_n(E)$ over any congruence family of elliptic curves over $\Q$ is $\sum_{d \mid n} d$.}
\end{remark}

For any given set $S$, the $p$-adic integrals in Theorem \ref{mainS} can be evaluated explicitly using Proposition~\ref{selmeratio} in \S\ref{Qporbitsec}.  In particular, when $S = \Z$, we recover Theorem \ref{main}.  

We note that unlike the results of 
\cite{bs2,bs3,bs4,bs5} for the 2-, 3-, 4-, and 5-Selmer groups, respectively, the average size of the $\phi$-Selmer group in Theorem~\ref{mainS}  can depend very much on the congruences defining the set~$S$.  However, we show that we may partition the set of non-zero integers into a countable union $\cup_{m=-\infty}^\infty T_m$ of sets $T_m$, where each $T_m$ is itself the union of countably many sets defined by congruence conditions, such that if $S\subset T_m$ is an acceptable set, then the average size of the $\phi$-Selmer group of $E_k\colon y^2=x^3+k$, as $k$ varies in $S$, depends only on $m$.  


More precisely, we define the {\it global Selmer ratio}  
\begin{equation}\label{defgsr}c(\phi_k) := \prod_p c_p(\phi_k)\end{equation}
to be the product of the local Selmer ratios. 
If $\phi$ is an $\ell$-isogeny, for some prime $\ell$, then the global Selmer ratio $c(\phi)$ is evidently a power of $\ell$.  The importance of the global Selmer ratio in the study of the Selmer groups $\Sel_\phi(E_k)$ is clearly seen in the following theorem.

\begin{theorem}\label{mavg}
For each $m\in\Z$, let $T_m:=\{k\in\Z \colon c(\phi_k)=3^m\}$. Then:
\begin{enumerate}
\item $\Z \setminus \{ 0 \}$ is the disjoint union of the $T_m$;
\item Each $T_m$ has positive density;
\item If $k\in T_m$, and $E_k(\Q)$ and $E_{-27k}(\Q)$ have trivial $3$-torsion, then $|\Sel_{\phi}(E_{k})|=3^m|\Sel_{\hat\phi}(E_{-27k})|$; 
\item If $S = T_m$, or if $S$ is any acceptable set contained in $T_m$, then the average size of $\Sel_\phi(E_k)$, $k\in S$, is equal to $1+3^m$, and the average size of $\Sel_{\hat\phi}(E_{-27k})$, $k\in S$, is equal to $1+3^{-m}$. 
\end{enumerate}
\end{theorem}

Thus the average size of $\Sel_\phi(E_k)$, $k \in S$, {\it is} independent of congruence conditions if we also fix the global Selmer ratio.
The densities of the sets $T_m$, to three decimal places, are tabulated in Table~1, along with the densities of the subsets $T_m^+$ (respectively, $T_m^-$) of $T_m$ consisting of its positive (respectively, negative) elements. We will describe the sets $T_m$ via explicit congruence conditions in \S\ref{Tmproof}.



\begin{table}
\[ \begin{array}{|r||l||l|l|} \hline m\: & \mu(T_m) & \mu(T^+_m) & \mu(T^-_m) \\ \hline\hline
-4\: & \;.000 & \;.000 & \;.000 \\ \hline
-3\: & \;.004 &  \;.004 & \;.000 \\ \hline
-2\: & \;.067 & \;.063 & \;.004 \\ \hline
-1\: & \;.295 & \;.231 & \;.063 \\ \hline
0\: & \;.399 &  \;.167 & \;.231 \\ \hline
1\: & \;.199&  \;.031 & \;.167 \\ \hline
2\: & \;.032 & \;.000 & \;.031 \\ \hline
3\: & \;.000 & \;.000 & \;.000 \\ \hline
4\: &\;.000 & \;.000 & \;.000\\ \hline
\end{array} \]
\caption{Densities of the sets $T_m$, $T_m^+$, and $T_m^-$
for $|m|\leq 4$.}
\end{table}

Using Theorem \ref{mavg} and the rigorous computation of the densities $\mu(T_m)$, we may obtain bounds on the limsup of the average {$3$-ranks} of the $\phi$- and $\hat\phi$-Selmer groups.  These immediately imply a bound on the limsup of the average 3-rank of the 3-Selmer group $\Sel_3(E_k)$, and hence a bound on the limsup of the average rank of $E_k(\Q)$ as $k$ varies in $\Z$ ordered by absolute value.  We prove the following theorem:

\begin{theorem}\label{rankthm}
 The $($limsup of the$)$ average rank of the elliptic curves $E_k\colon y^2=x^3+k$, $k\in \Z$, is less than $1.29$.  
\end{theorem}



Theorem~\ref{rankthm} immediately implies that a positive proportion of elliptic curves $E_k\colon y^2=x^3+k$ must have rank~0 or~1.  
To produce positive proportions of curves having the individual ranks 0 or 1, we may make use of the fact in Theorem~\ref{mavg}(iii) that for 100\% of $k \in T_m$, the 3-rank of the $\phi$-Selmer group of $E_k$ is $m$ more than that of its $\hat\phi$-Selmer group.  This means, in particular, that 100\% of curves in $T_0$ have even $3$-Selmer rank (Proposition \ref{rankineq}(ii)). 
Since the average sizes of both the $\phi$ and $\hat\phi$-Selmer groups are equal to 2,  we conclude that at least 1/2 of all curves in $T_0$ must in fact have 3-Selmer rank 0.  Hence we obtain the following theorem.

\begin{theorem}\label{posrank0}
  At least $19.9\%$ of all elliptic curves $E_k\colon y^2=x^3+k$, $k\in \Z$, have rank~$0$.
\end{theorem}

Conditionally on the finiteness of the 3-primary part of the Shafarevich--Tate group, we may obtain a similar result for rank 1 elliptic curves by using instead the set $T_1 \cup T_{-1}$, for which the 3-Selmer parity is always odd.  

\begin{theorem}\label{rank1}
Assume for all $k \in T_1\cup T_{-1}$ that the $3$-primary part of $\Sha(E_k)$ is finite.  Then at least 
$41.1\%$ of all elliptic curves $E_k\colon y^2=x^3+k$, $k\in \Z$, have rank~$1$.
\end{theorem}

Using forthcoming results of Skinner (private communication),  we may be able to deduce this last result unconditionally.

Combining Theorems~\ref{posrank0} and \ref{rank1}, we obtain:

\begin{corollary}\label{rank0or1}
At least $61\%$ of all elliptic curves $E_k\colon y^2=x^3+k$, $k\in \Z$, have rank~$0$ or~$1$.
\end{corollary}
Thus the majority of curves $E_k\colon y^2=x^3+k$ have rank 0 or 1. 

\begin{remark}\label{ruth}
{\em Theorems \ref{rankthm}, \ref{posrank0}, and \ref{rank1} and Corollary~\ref{rank0or1} can be improved slightly if we combine our results with those in Ruth's thesis \cite{ruth}, where it is shown that the average size of $\Sel_2(E_k)$, $k \in S$, is  at most~3.  Indeed, this gives a rank bound of 4/3 on any $T_m$, which turns out to be a better bound than ours for $m > 1$.  The improvements are small since even $T_2$ has very small density: incorporating Ruth's result, we can show that the average rank of $E_k$ is at most 1.21, and the proportion of rank 0 curves is at least~23.2\%.}
\end{remark}

Finally, we note that Theorem~\ref{mavg} implies:

\begin{theorem}\label{big3sel}
For each $m\geq 0$, a positive proportion of elliptic curves $E_k\colon y^2=x^3+k$, $k\in\Z$, have $3$-Selmer rank  equal to $m$.
\end{theorem}

Indeed, the average size of $\Sel_{\hat\phi}(E_k)$ is $1+3^{-m}$ for $k\in T_m$ and, therefore, a positive proportion of $k\in T_m$ satisfy $|\Sel_{\hat\phi}(E_k)|=1$.  By Theorem~\ref{mavg}(iii), we then have $|\Sel_{\phi}(E_k)|=3^m$ and thus $|\Sel_3(E_k)|=3^m$ for most of these $k$. 
Taking into account Remark~\ref{ruth} and Selmer parity~\cite{dokchitser} we see that, for even $m$, Theorem~\ref{big3sel} holds not just for $\Sel_3(E_k)$ but for $\Sha(E_k)[3]$ as well.  

Our methods involve the connection between $\phi$-Selmer groups and binary cubic forms.  This connection was first studied by Selmer himself \cite{selmer}, and later by Cassels~\cite{cas2}.  The rational theory was thoroughly treated by Satg\'e \cite{satge}, where he studied the $\phi$-Selmer group from the point of view of cubic fields.  Later, Liverance~\cite{Liv} studied these Selmer groups using the classical invariant theory of binary cubics. 

The boundedness of the average size of the $\phi$- and $\hat\phi$-Selmer groups---and thus the rank---of the elliptic curves $E_k \colon y^2=x^3 +k$, $k\in\Z$, was first demonstrated by Fouvry~\cite{fouvry}, who used Satg\'e's results to reduce the boundedness of the average rank to the theorem of Davenport and Heilbronn~\cite{dh} on the boundedness of the average size of the 3-torsion subgroups of the class groups of quadratic fields.  Fouvry's method thus implicitly used  binary cubic forms, as Davenport and Heilbronn's proof on the mean size of the 3-torsion of the class groups of quadratic fields used a count of integral binary cubic forms to count cubic fields of bounded discriminant, together with class field theory to transform this count to one about 3-torsion elements in class groups of quadratic fields. 

Recently, in \cite{bv} a more direct proof of Davenport and Heilbronn's theorem on 3-torsion elements in class groups in quadratic fields was given.  This proof used a count of {\it integer-matrix} binary cubic forms $ax^3+3bx^2y+3cxy^2+dy^3$ ($a,b,c,d\in\Z$), together with a {direct}  correspondence between 3-torsion ideal classes in quadratic fields and integer-matrix binary cubic forms as studied in \cite{eisenstein,hcl1} (see Section~\ref{secorbits} for a description of this correspondence over any Dedekind domain).  This suggested to us that perhaps a natural discriminant-preserving map from $\phi$-Selmer groups $\Sel_\phi(E_k)$ to orbits of integer-matrix binary cubic forms could be constructed.  This construction is the key to our method, and is the subject of Sections~\ref{construction}--6.  

With this direct map in hand, the problem of determining the average size of the $\phi$-Selmer group is reduced to counting the appropriate set of integer-matrix binary cubic forms of bounded discriminant.  For this, the counting method of Davenport~\cite{Davenport2} and a suitable adaptation of the sieve methods of \cite{dh} and \cite{bs2} may be applied to obtain the optimal upper and lower bounds, and this is carried out in Section~\ref{proofmain}.  This sieve renders the final answer as a product of local densities, and we prove that the density at the $p$-adic place for a given elliptic curve $E_k \colon y^2=x^3+k$ is given precisely by the local Selmer ratio $c_p(\phi)$ as defined by (\ref{deflsr}).  For $k$ varying over an acceptable set $S$, this yields Theorem~\ref{mainS} (see \S\ref{final}).  

We may evaluate the local Selmer ratios by relating them to Tamagawa numbers and using Tate's algorithm (Proposition \ref{selmeratio}).  Setting $S=\Z$ then yields Theorem~\ref{main}, while setting $S=T_m$, and applying a formula of Cassels~\cite{cas} on the global Selmer ratio, yields Theorem~\ref{mavg} (see \S\S\ref{final}--\ref{Tmproof}).
Finally, the results on average rank and on positive proportions of rank~0 and 3-Selmer rank~1 curves, as in Theorems~\ref{rankthm}--\ref{rank1} and Corollary~\ref{rank0or1}, are deduced in Section~\ref{ranks}. 

We note that the connection between binary cubic forms and Selmer groups also has an interpretation in terms of Lie groups.  In the language of Vinberg theory, the representation of $\SL_2$ on binary cubic forms arises from a $\Z/3\Z$-grading of the Lie algebra of $G_2$.  This paper is one example of recent work connecting representations arising from Vinberg's theory to Selmer groups of Jacobians of algebraic curves.  In this context, we believe that our theorems above give the first results for the exact and finite average size of a
Selmer group associated to an isogeny that is not the
multiplication-by-$n$ isogeny.  For results and conjectures about the
latter see, e.g., \cite{bg, thorne1, thorne2, bs2, HB, Kane,KMR, PR, SD}.  A beautiful treatment of the connection between algebraic curves and Vinberg theory, using versal deformations, has been given by Thorne~\cite{thorne1}.  

One may ask about the $\phi$-Selmer group in families of quadratic twists $E_{km^3}\colon y^2=x^3+km^3$, for fixed $k$ and $m$ varying, as well as in families of cubic twists $E_{km^2}\colon y^2=x^3+km^2$.  These are very sparse subsets of all curves $E_k$, and so our results in Theorem~\ref{mainS} do not apply.  Nevertheless, in a forthcoming paper~\cite{bkls}, it is shown that the analogue of Theorem~\ref{mainS} continues to hold even in these families of quadratic twists; in fact, the formula (\ref{mainSformula}) for the average size of the $\phi$-Selmer group holds more generally in any family of quadratic twists of an elliptic curve curve (not necessarily of $j$-invariant 0) having a rational 3-isogeny, leading to a proof of the boundedness of the average rank and a positive proportion having rank 0 also in such quadratic twist families.  We note that the analogue of the formula in Theorem~\ref{mainS}  holds also for the above cubic twist families~\cite{BEKS?}, although in most such cases the product (\ref{mainSformula}) does not converge and we find that the average $\phi$-Selmer group sizes for curves in these families are unbounded, {\it except} in the case where $k=1$. 
 
Finally, in another forthcoming work, Kriz and Li (private communication) give a very different proof that a positive proportion of the curves $y^2=x^3+k$ have rank 0 (resp.\ 1); they show that around 10\% of these curves have rank 0, and a similar proportion have rank 1.  Their $p$-adic methods are completely different from ours, and in particular, their rank 1 results are unconditional and do not rely on the recent work of Skinner, W.\ Zhang, and others on converse theorems.  It will be an interesting future question to see if there is a way to increase the rank proportions by combining our method with theirs.
\section{Integer-matrix binary cubic forms over a Dedekind domain}\label{secorbits}

Let $V^*(\Z) = \Sym^3 \Z^2$ be the lattice of integer-coefficient binary cubic forms, i.e., forms $f(u,v) = au^3 + bv^2y + cuv^2 + dv^3$ with $a,b,c,d\in\Z$.  To ease notation, we write $f(u,v) = [a,b,c,d]$.  The group $\GL_2(\Z)$ acts naturally on $V^*(\Z)$ by linear change of variable, and the {\it discriminant} 
\[\Disc(f) = b^2c^2 -4ac^3 - 4b^3d - 27a^2d^2 +18abcd\]
is $\SL_2(\Z)$-invariant.  

In this paper, the more fundamental object will be the dual lattice $V(\Z) = \Sym_3 \Z^2$ consisting of {\it integer-matrix} binary cubic forms, i.e., forms $f(u,v) = [a,3b,3c,d]$ with $a,b,c,d \in \Z$.  The lattice $V(\Z)$ has its own {\it $($reduced$)$ discriminant} $$\Delta(f) = -\frac{1}{27}\Disc(f).$$  
We note that the action of $\GL_2(F)$ on the space $V(F):=V(\Z)\otimes F$ of binary cubic forms with coefficients in $F$ satisfies
\[\Delta(g\cdot f) = \det(g)^6 \Delta(f)\]
for all $g \in \GL_2(F)$ and $f \in V(F)$. For any ring $R$ and $d \in R$, we write $V(R)_d$ for the set of $f \in V(R):=V(\Z)\otimes R$ with $\Delta(f) = d$.

In this section, we classify the orbits of $V(D)$, under the action of $\SL_2(D)$, for an arbitrary Dedekind domain $D$.  We assume that $D$ does not have characteristic 2 or 3.  In later sections we will apply our classification to the case where $D$ is a field, $\Z_p$, or $\Z$.  Our result is a generalization of \cite[Theorem~13]{hcl1}, and is proved in the same way.  



\begin{theorem}\label{hcl1}
Let $D$ be a Dedekind domain of characteristic not $2$ or $3$, and let $k\in D$ be any nonzero element.  Let $F$ be the fraction field of $D$, and let $S:=D[z]/(z^2 - k)$ and $K := F[z]/(z^2 - k)$.  
Then there is a bijection between the orbits of $\SL_2(D)$ on $V(D)_{4k}$ and equivalence classes of triples $(I, \delta, s)$, where $I$ is a fractional $S$-ideal, $\delta \in K^*$, and  $s \in F^*$, satisfying the relations $I^3 \subset \delta S$, $N(I)$ is the principal fractional ideal $sD$ in $F$, and $N(\delta) = s^3$ in $F^*$.  Two triples $(I, \delta, s)$ and $(I', \delta', s')$ are equivalent if there exists $\kappa \in K^*$ such that $I '=  \kappa I$, $\delta' = \kappa^3 \delta$, and $s' = N(\kappa)s$.  Under this correspondence, the stabilizer in $\SL_2(D)$ of $f \in V(D)_{4k}$ is isomorphic to $S(I)^*[3]_{N = 1}$, where $S(I)$ is the ring of endomorphisms of $I$. 
\end{theorem}


\begin{remark}\label{valid}
{\em We call triples $(I, \delta, s)$ satisfying the relations above {\it valid triples}.  }
\end{remark}

\begin{proof}
We describe here the explicit bijection, but refer  readers to \cite{hcl1} for the details of the proof.  First, note that $S = D + D\tau$, with $\tau$ the image of $x$ in $S$.   
Given a valid triple $(I, \delta, s)$, since $N(I)$ is the principal $D$-ideal $sD$, the projective $D$-module $I$ of rank 2 is free, and so we may write $I = D\alpha  + D\beta$ for some $\alpha, \beta\in I$.  
 Because $I^3\subset \delta S$, we have
\begin{equation}\label{trilinearform1}
\begin{array}{ccl}
\alpha^3 &=& \delta(e_0 + \tau a)\\
\alpha^2 \beta &=& \delta(e_1 + \tau b)\\
\alpha \beta^2 &=& \delta(e _2+ \tau c)\\
 \beta^3 &=& \delta(e_3 + \tau d),
\end{array}   
\end{equation}
for some integers $a,\, b,\, c,\, d,e_i\in\Z$.  Then  corresponding to $(I,\delta, s)$ is the binary cubic form $f = [a, 3b,3c,d]$, which has discriminant $\Delta(f) = 4k$.   In more coordinate-free terms, $f$ is the symmetric trilinear form 
\begin{equation}\label{trilinearform}\frac1\delta\times:I\times I\times I\to S/D\cong D\tau.\end{equation}
We obtain an $\SL_2(D)$-orbit of symmetric trilinear forms (integer-matrix binary cubic forms) over~$D$ by taking the symmetric $2\times2\times2$ matrix representation of this form with respect to any ordered basis $\langle\alpha,\beta\rangle$ of $I$ that gives rise to the basis element $s(1\wedge \tau)$ of the top exterior power of $I$ over $D$. This normalization deals with the difference between $\SL_2(D)$- and $\GL_2(D)$-orbits. The stabilizer statement follows because elements in $S(I)^*[3]_{N = 1}$ are precisely the elements of $K^*_{N=1}$ that preserve the map (\ref{trilinearform}). 
\end{proof}
When $D$ is a field, so that $D = F$, the previous result simplifies quite a bit. Let us write $(K^*/K^{*3})_{N =1}$ to denote the kernel of the norm map $K^*/K^{*3} \to F^*/F^{*3}$, and $(\Res^K_F  \mu_3)_{N = 1}$ for the kernel of the norm map $\Res^K_F \, \mu_3 \to \mu_3$.  

\begin{corollary}\label{bij}
There is a bijection between the set of $\SL_2(F)$-orbits on $V(F)_{4k}$ and $\left(K^*/K^{*3}\right)_{N =1}$. Moreover, the stabilizer of any $f \in V(F)_{4k}$  in $\SL_2(F)$ is isomorphic to $(\Res_F^K \, \mu_3)_{N = 1}$.     
\end{corollary}
\begin{proof}
Both statements follow from taking $D = F$ in the previous theorem.  The bijection sends $\delta \in (K^*/K^{*3})_{N = 1}$ to the orbit of binary cubic forms corresponding to the triple $(K, \delta, s)$, where $s$ is any choice of cube root of $N(\delta)$.      
\end{proof}

\begin{remark}\label{explicitcubic}{\em 

 Explicitly, if $\delta = a + b\tau \in K^*$, then the  corresponding cubic form is $f=[ak,3bk,3a,b].$}
\end{remark}

Another case of interest is when $D = \Z_p$ and $S$ is the maximal order in $K$:
\begin{proposition}\label{intorb}
Let $p$ be a prime and assume $S = \Z_p[z]/(z^2 - k)$ is the maximal order in $K = \Q_p[z]/(z^2 - k)$.  Then the set of $\SL_2(\Z_p)$-orbits on $V(\Z_p)_{4k}$ is in bijection with the unit subgroup $(S^*/S^{*3})_{N =1} \subset (K^*/K^{*3})_{N = 1}$.  Every rational $\SL_2(\Q_p)$-orbit of discriminant $4k$ whose class lies in this unit subgroup contains a unique integral $\SL_2(\Z_p)$-orbit.   The stabilizer in $\SL_2(\Q_p)$ of an element in $V(\Z_p)_{4k}$ is equal to its stabilizer in $\SL_2(\Z_p)$. 
\end{proposition}

\begin{proof}
Indeed, every $S$-ideal $I$ is principal in this case.  From the $\Z_p$-version of Theorem \ref{hcl1}, it follows that any integral orbit corresponds to a triple $(S, \delta,s)$ with $\delta \in S^*$. The stabilizer statement follows because $K^*[3]_{N=1}=S^*[3]_{N=1}$ when $S$ is the maximal order in $K$. 
\end{proof}

Finally, we will need the following result on stabilizers of integer-matrix binary cubic forms over $\Z$.  
Let $G = \SL_2$ and for any ring $R$ and $f\in V(R)$, write $\Aut_R(f)$ for $\Stab_{G(R)}(f)$. 

\begin{proposition}\label{stabs}
Suppose $f \in V(\Z)_{4k}$ is an integer-matrix binary cubic form of discriminant $4k$.  Then 
\[|\Aut_\Q(f)| \sum_{ f' \in O(f)} |\Aut_{\Z}( f' )|^{-1} = \prod_p  |\Aut_{\Q_p}(f)| \sum_{f'_p \in O_p(f)} |\Aut_{\Z_p}(f'_p)|^{-1},\]
where $O(f)$ is a set of representatives for the $G(\Z)$-orbits on $V(\Z)_{4k}$ contained in the $G(\Q)$-orbit of~$f$, and similarly $O_p(f)$ is a set of representatives for the $G(\Z_p)$-orbits on $V(\Z_p)_{4k}$ contained in the $G(\Q_p)$-orbit of $f$.   
\end{proposition}

\begin{proof}
We follow the proof of \cite[Proposition~8.9]{bg}.  Let $K = \Q[z]/(z^2 - k)$, and fix a representative $\delta \in K^*$ of the class in $(K^*/K^{*3})_{N=1}$ corresponding to $f$.  Let $m(\delta)$ be the number of ideals $I$ of $S = \Z[z]/(z^2 - k)$ satisfying $I^3 \subset \delta S$ and the ideal equality $N(I)^3 = N(\delta)$.  Similarly, let $m_p(\delta)$ be the number of ideals $I_p$ of $S_p := S \otimes \Z_p$ with $I_p^3 \subset \delta S_p$ and $N(I_p)^3 = N(\delta)$.  Note that for all but finitely many $p$, $S_p$ is the maximal order and $\delta$ is a unit, so $m_p(\delta) =1$ for all but finitely many $p$.  

Since a lattice is determined by its local completions, and since a collection of local ideals produces a global ideal with the desired properties, we have      
\begin{equation}\label{lg} m(\delta) = \prod_p m_p(\delta).\end{equation}
Let $s$ be any cube root of $N(\delta)$.  Then the triple $(I, \delta, s)$ is valid and hence corresponds to an integer-matrix binary cubic form $f_I \in V(\Z)_{4k}$ mapping to the rational orbit of $\delta$.  If $s'$ is another choice of cube root, then $s' = s\zeta_3$ and $(I,\delta, s)$ is equivalent to $(I,  \delta, s')$.  A triple $(I,\alpha,s)$ is in the same integral orbit as $(cI, \alpha,s)$, for some $c \in K^\times$, exactly when $c^3 = 1$ and $N(c) = 1$.  On the other hand, the ideals $I$ and $cI$ are equal if $c$ is a unit in the ring of endomorphisms $S(I)$ of $I$.  Thus, the number of distinct ideals giving the same integral orbit is the size of the group $K^*[3]_{N = 1} /S(I)^*[3]_{N =1}$.  

We have $K^*[3]_{N = 1} \cong \Aut_{\Q}(f)$ and $S(I)^*[3]_{N =1} \cong \Aut_{\Z}(f_I)$ by Theorem \ref{hcl1}.  Thus, the number of distinct ideals associated to the integral orbit of $f_I$ is $|\Aut_{\Q}(f)|/|\Aut_{\Z}(f_I)|$.  We conclude that 
\[m(\delta) = \sum_{f_I \in O(f)} |\Aut_{\Q}(f)|/|\Aut_{\Z}(f_I)|.\]
The same reasoning implies an analogous formula for $m_p(\delta)$, with global stabilizers replaced by local stabilizers.  The proposition now follows from (\ref{lg}).      
\end{proof}

\section{The elliptic curves $E_k$ and orbits of binary cubic forms over a field}\label{construction}

Let $F$ be a field of characteristic not 2 or 3, and let $k \in F$ be non-zero.  Recall that the elliptic curve  
\begin{equation}E_k \colon y^2 = x^3  + k\end{equation}
admits a 3-isogeny $\phi \colon E_k \to E_{-27k}$ defined over $F$, and a dual 3-isogeny $\hat\phi \colon E_{-27k}\to E_k$.  

In this section, we describe the connection between the isogenies $\phi$ and $\hat\phi$ and binary cubic forms (i.e., symmetric trilinear forms) over $F$. 

\subsection{Galois cohomology of the 3-isogeny kernel and field arithmetic}

Important in the study of $\phi$- and $\hat\phi$-descents on the curves $E_k$ and $E_{-27k}$ is the pair of ``mirror'' quadratic \'etale $F$-algebras
\[K = F[z]/(z^2 - k) \hspace{ 4mm} \mbox{and} \hspace{4mm} \hat K = F[z]/(x^2 + 27k).\] 
The following result connects the arithmetic of the elliptic curves $E_{-27k}$ and $E_{k}$ to the arithmetic of $K$ and $\hat K$:

\begin{proposition}\label{cassels}
There is an isomorphism of  group schemes
\[E_{-27k}[\hat \phi] \cong \ker\left(\Res_F^K \mu_3 \to \mu_3\right),\]  
and an induced isomorphism 
\[H^1(G_F, E_{-27k}[\hat\phi]) \cong \left(K^*/K^{*3}\right)_{N = 1},\]
where $\left(K^*/K^{*3}\right)_{N= 1}$  denotes the kernel of the norm $N\colon K^*/K^{*3} \to F^*/F^{*3}$.
\end{proposition}

\begin{proof}
Duality gives a non-degenerate pairing
\[\langle \;\,, \; \rangle \colon  E_{-27k}[\hat\phi] \times E_k[\phi] \to \mu_3.\]
Since $E_k[\phi](\bar F)$ becomes a trivial Galois module when restricted to $G_K$, this induces an injective homomorphism of group schemes $\iota \colon E_{-27k}[\hat\phi] \to \Res^K_F \mu_3$,
given on points by \[P \mapsto (\langle P, Q_1 \rangle, \langle P, Q_2\rangle),\]
where $Q_1$ and $Q_2$ are the non-trivial points of $E_k[\phi]$.  The image of $\iota$ is precisely the kernel of the norm map $\Res^K_F \mu_3 \to \mu_3$, giving the desired isomorphism.

From Kummer theory and  the long exact sequence attached to  
\[0 \to E_{-27k}[\hat\phi] \to \Res_F^K \, \mu_3  \to \mu_3 \to 0,\]
we obtain the isomorphism $H^1(G_F, E_{-27k}[\hat\phi]) \cong \left(K^*/K^{*3}\right)_{N = 1}$.
\end{proof}

\begin{remark}\label{kummer}
{\em The sequence $0 \to E_{-27k}[\hat\phi] \to E_{-27k} \to E_k \to 0$ induces a Kummer map 
\[\partial \colon E_k(F) \to H^1(G_F, E_{-27k}[\hat\phi]) \cong \left(K^*/K^{*3}\right)_{N = 1},\] which can be described explicitly as follows.  If $(x,y)\notin E_k[\phi](F)$,  then  $\partial((x,y)) = y - \tau$, where $\tau$ is the image of $x$ in $K = F[z]/(z^2 - k)$.  If $P = (0,\pm \sqrt k) \in E_k[\phi]$, then $\partial(P) =  \pm 1/2\tau$.  See \cite[\S 15]{cas}.        }
\end{remark}

\begin{remark}{\em 
Of course, the analogues of all these results hold also when $\hat\phi$ is replaced with~$\phi$ and $K$ is replaced with~$\hat K$; these analogues are obtained simply via the change of variable $k\mapsto -27k$.}
\end{remark}


%

\subsection{Connection to binary cubic forms}

We may now compare Corollary \ref{bij} and Proposition \ref{cassels}.  This immediately 
yields the following canonical bijection:

\begin{theorem}\label{H1bij}
There is a bijection between $H^1(G_F, E_{-27k}[\hat\phi])$ and the $\SL_2(F)$-orbits on $V(F)_{4k}$.    Moroever, the stabilizer in $\SL_2(F)$ of any $f \in V(F)_{4k}$ is isomorphic to $E_{-27k}[\hat\phi](F)$. 
\end{theorem}
Let $V(F)^\sol$ denote the set of binary cubic forms $f(x,y)\in V(F)$ that correspond under the bijection of Theorem~\ref{H1bij} to classes in the image of the Kummer map $\partial \colon E_k(F) \to H^1(G_F, E_{-27k}[\hat\phi])$ for some $k \in F$.  If $ f\in V(F)^\sol$, then we also say that $f$ is {\it soluble}.  We write $V(F)^\sol_{D}$ for the elements of $V(F)^\sol$ having discriminant $D$. 

\begin{corollary}\label{soluble}
There is a natural bijection between  the $\SL_2(F)$-orbits on
  $V(F)^\sol_{4k}$ and the elements of the group $E_k(F)/\hat\phi(E_{-27k}(F))$.  
  Under this bijection, the identity element of  $E_k(F)/\hat\phi(E_{-27k}(F))$ corresponds to the unique $\SL_2(F)$-orbit of reducible binary cubic forms in $V(F)^\sol_{4k}$, namely the orbit of $f = [k,0,3,0]$. 
\end{corollary}
 \begin{proof}
The second part of the corollary follows from Remark \ref{explicitcubic}, by taking $\delta=a=1$ and $b=0$.  
\end{proof}

The use of the term ``soluble'' comes from the following fact:


\begin{proposition}\label{solubility}
A binary cubic form $f(x,y)\in V(F)_{4k}$ corresponds to an element $\delta$ in the image of $\partial \colon E_k(F) \to H^1(G_F, E_{-27k}[\hat\phi])\cong\left(K/K^{*3}\right)_{N = 1}$ if and only if the curve $C_f \colon z^3 = f(x,y)$ in $\P^2$ has an $F$-rational point.  
\end{proposition}

\begin{proof}
Suppose $f\in V(F)_{4k}$, and let $\delta\in K^*$ be a representative element in $\left(K/K^{*3}\right)_{N = 1}$
corresponding to the $\SL_2(F)$-orbit of $f$ under the bijection of Corollary~\ref{bij}.  


If $\delta$ represents the image of a point $(x,y)\in E_k(F)$ under $\partial$,
then, by Remarks~\ref{explicitcubic} and \ref{kummer}, the form $f$ is $\SL_2(F)$-equivalent to 
$[1,3y,3k,yk]$, which represents a cube (namely, $1$). 

Conversely, if there exist $u,v,z \in F$ such that $z^3 = f(u,v)$, then by scaling $u$ and $v$ if necessary we may assume that $z=1$. Therefore, we have by (\ref{trilinearform1}) or (\ref{trilinearform}) 
that 
\begin{equation}\label{conversely}(u+v\tau)^3 = \delta(g(u,v) + f(u,v)\tau)=\delta(g(u,v)+\tau)
\end{equation}
where $g$ is a binary cubic form over $F$, and so $g(u,v)\in F$. Writing $N(\delta)=s^3$, and then taking norms of both sides of (\ref{conversely}), we obtain $N(u+v\tau)^3=s^3(g(u,v)^2-k)$.  Thus the point $(x,y)=(N(u+v\tau)/s,g(u,v))$ lies on $E_k$, and 
this point $(x,y)$ then maps to the class of $\delta$ in $\left(K/K^{*3}\right)_{N = 1}$ under $\partial$. 
\end{proof}

\begin{remark}{\em 
Recall that a $\hat\phi$-covering is a map of curves $C \to E_k$ over $F$ which is a twist of $\hat\phi$.  By descent, the group $H^1(G_F, E_{-27k}[\hat\phi])$ is in bijection with isomorphism classes of $\hat\phi$-coverings.  To construct the $\hat\phi$-covering corresponding to $\delta \in \left(K/K^{*3}\right)_{N = 1}$, take any $s \in F^*$ such that $N(\delta) = s^3$, and let $f$ be the corresponding binary cubic form over $F$ under the bijection of Corollary~\ref{bij}.  Then we take $C_f$ to be the curve $z^3=f(x,y)$ in $\P^2$ (whose Jacobian is easily computed to be $E_k$), and the $\hat\phi$-covering map $C_f \to E_k$ corresponding to the class $\delta$ is given explicitly by  
\[(u\colon v\colon z) \mapsto ((u^2 - kv^2)/s, g(u,v)),\]  
where $g(u,v)$ is the cubic polynomial in the preceding proof.}
\end{remark}

\begin{remark}{\em One can also construct the $\hat\phi$-covering using the invariant theory of binary cubics. If $f  = [a,3b,3c,d] \in V(F)$, then the covariants of $f$ are generated by the discriminant $\Delta$, the scaled Hessian
\begin{equation}
\label{hessian}
  h(x,y) = \frac{1}{36}(f_{xx} f_{yy} - f_{xy}^2)
  = (ac-b^2) x^2 + (ad-bc) xy + (bd-c^2),
\end{equation}
and the Jacobian derivative of $f$ and $h$,
\begin{equation}
\label{jacderiv}
  g(x,y) = \frac{\partial(f,h)}{\partial(x,y)}
         = f_x h_y - f_y h_x,
\end{equation}
which is a cubic polynomial in $x,y$ whose coefficients are cubic polynomials
in $a,b,c,d$.
The cubic $f$ and its covariants $g,h$ are related by the syzygy\footnote{
  These results are classical, at least for $F = \C$,
  and they go back at least to Hilbert \cite[pp.68--69]{Hilbert};
  see also Schur's treatment \cite[II, \S8, Satz~2.24 on p.~77]{Schur}.
  The syzygy (\ref{syzygy}) can be verified by direct computation,
  though it is easier to check it for one choice of
  $f$ without repeated factors and then use the fact that
  $\SL_2(\Fbar)$ acts transitively on such~$f$.
  For example, $f=x^3-y^3$ gives $\Delta=1$, $h=-xy$ 
  and $g = -3(x^3+y^3)$, 
  reducing the syzygy (\ref{syzygy}) to
  the identity $(x^3+y^3)^2 - (x^3-y^3)^2 = 4(xy)^3$.
  }
\begin{equation}
\label{syzygy}
(g/3)^2 -   \Delta(f) f^2  + 4h^3 = 0.
\end{equation}
This gives us a degree-$3$ map from the genus-1 curve
$$
C_f \colon z^3 = f(x,y)
$$
to~$E_k$ with $k = \Delta/4$:
divide both sides of (\ref{syzygy}) by $4z^6$ and solve for $(g/2z^3)^2$
to obtain
\begin{equation}
\label{jac_map}
  \left(\frac16\frac{g}{z^3}\right)^2
  = \left(-\frac{h}{z^2}\right)^3 + \frac{\Delta(f)}{4}.
\end{equation}
Our curve $E_{-27k}$ is the special case of $C_f$ where
$f(x,y) = kx^3 + 3x y^2$;
we then see that $\Delta(f) = 4k$,
and the map $(z, y/x) \mapsto (-h/z^2, g/6z^3)$ to $E_k$
recovers our formula (\ref{phihatdef}) for the \hbox{3-isogeny $\hat\phi$.}} 
\end{remark}

\section{Soluble orbits over local fields}

When $F$ is a local field, we can give explicit formulas for the number of soluble $\SL_2(F)$-orbits of binary cubic forms of discriminant $4k$, i.e., the size of the group $E_k(F)/\hat\phi(E_{-27k}(F))$.  Since $|E_{-27k}[\hat\phi](F)|$ is 3 or 1 depending on whether $-3k$ is a square in $F$, it is equivalent to give formulas for the (local) Selmer ratios
\[c(\hat\phi_k) = \dfrac{|\coker(E_{-27k}(F) \to E_k(F))|}{|\ker(E_{-27k}(F) \to E_k(F))|}.\]
We do this below for the local fields $\Q_p$, $\R$, and $\C$, though there are similar formulas for any finite extension of $\Q_p$ and for equicharacteristic local fields such as $\F_p((t))$.  If $F = \Q_p$, $\R$, or $\C$, we use the notation $c_p(\hat\phi_k)$, with $p \leq \infty$, to match with the introduction.  We state the result for $c_p(\phi_k)$, the Selmer ratio of the original isogeny $\phi_k\colon E_k\to E_{-27k}$. 

\subsection{Orbits over $\Q_p$}\label{Qporbitsec}
%
Suppose $F = \Q_p$.  We first determine when $E_k/\Q_p$ has good reduction.

\begin{lemma}\label{reduction}
Assume $k \in \Z_p$ is sixth-power-free.  If $p > 3$, then $E_k/\Q_p$ has good reduction if and only if $p \nmid k$.  If $p = 3$, then $E_k/\Q_3$ has bad reduction.  If $p = 2$, then $E_k/\Q_2$ has good reduction if and only if
$k \equiv 16 \pmod {64}$.
\end{lemma}

\begin{proof}
This follows from Tate's algorithm.  If $k \equiv 16 \pmod {64}$, a model with good reduction at~$2$
is ${y'}^2 + y' = {x'}^3 + a_6$ where $k = 64a_6 + 16$ and
$(x,y) = (4x', 8y' + 4)$.
\end{proof}

Next, we express $c_p(\phi_k)$ in terms of the Tamagawa numbers of $E_k$ and~$E_{-27k}$ (cf.\ \cite[Lemma~3.8]{schaefer}).
\begin{proposition}\label{tamag}
If $k \in \Z_p$ is sixth-power-free, then 
\[ c_p(\phi_k) = \dfrac{c_p(E_{-27k})}{c_p(E_k)}\times \begin{cases}
3 & \mbox{if } p = 3 \mbox{ and } 27 \mid k; \mbox{ and}\\
1 & \mbox{otherwise},
\end{cases}\]
where $c_p(E) = |E(\Q_p)/E_0(\Q_p)|$ denotes the Tamagawa number of $E$.  
\end{proposition}
\begin{proof}
  Let $\omega_k$ and
  $\omega_{-27k}$ denote N\'eron differentials on $E_k$ and $E_{-27k}$,
  respectively.  Then for some $b,b'\in\Q$, we have $\omega=b\cdot \frac{dx}{y}$ and
  $\omega_{-27k} = b'\cdot\frac{dX}{Y}$ on $E_k\colon y^2=x^3+k$ and $E_{-27k}\colon Y^2=X^3-27k$,
  respectively.  The model $y^2 = x^3 + k$ is minimal except if $k \equiv 16 \pmod {64}$, in which case the model given in the proof of Lemma \ref{reduction} is minimal.  In either case, we compute that $b=b'$ when $27\nmid k$ and
  $b'=3b$ when $27\mid k$.  Since
\[\phi^\ast\left(b\cdot\frac{dX}{Y}\right)=\frac{\displaystyle{b\cdot d\left(\frac{x^3+ 4k}{x^2}\right)}}{\frac{\displaystyle{y(x^3-8k)}}{\displaystyle{x^3}}}
=\frac{b\cdot\left(1-\displaystyle{\frac{8k}{x^3}}\right)dx}{\displaystyle{\frac{y(x^3-8k)}{x^3}}} = b\cdot\frac{dx}{y},\]
we have $\phi^\ast\omega_{-27k}=a \omega_{k}$, where $a=1$ if $27\nmid k$
and $a = 3$ if $27\mid k$.

We then compute 
$$\int_{\hat\phi(E_{k}(\Q_p))} |\omega_{-27k}|_p \,\,=\,\,\frac1{|E_{k}[\phi](\Q_p)|}\cdot\int_{E_{k}(\Q_p)}|\phi^\ast\omega_{-27k}|_p
\,\,=\,\,\frac1{|E_{k}[\phi](\Q_p)|}\cdot |a|_p\cdot \int_{E_{k}(\Q_p)}|\omega_{k}|_p.$$
Therefore,
$$c_p(\phi_k) = \frac{|(E_{-27k}(\Q_p)/\phi(E_{k}(\Q_p)))|}{|E_{k}[\phi](\Q_p)|} 
\,\,=\,\,
\frac{\int_{E_{-27k}(\Q_p)}|\omega_{-27k}|_p}{|a|_p\int_{E_{k}(\Q_p)}|\omega_{k}|_p}\,\,=\,\,
\frac1{\;\,|a|_p}\cdot\frac{c_p(E_{-27k})}{c_p(E_{k})}$$ as desired.
\end{proof}
We use Tate's algorithm \cite{tate} to compute the ratios of Tamagawa numbers in Proposition~\ref{tamag}, and hence the local Selmer ratios $c_p(\phi_k)$.  In particular, we find that the $c_p(\phi_k)$ are determined by congruence conditions on $k$:

\begin{proposition}\label{selmeratio}
Let $k\in\Z$ be sixth-power-free, and let $\chi_K$ denote the quadratic character attached to $K = \Q(\sqrt k)$.  If $p \neq 3$, then 
\[c_p(\phi_k) = 
\begin{cases}
3^{-\chi_K(p)} & \mbox{if $p \equiv 2$ $($mod $3)$ and  } v_p(4k) \in \{2 ,4\}; \\
1 & \mbox {otherwise}.
\end{cases}\]
If $p = 3$, write $k = 3^{v_3(k)}k_3$.  Then

\begin{equation}\label{selratioat3}
c_p(\phi_k) \;=\; \left\{\begin{array}{ll}
9 & \mbox{if } v_3(k) =5, \mbox{and  $k_3 \equiv 2$ {\rm{(mod 3)}};}\\ [.1in]
3  &  \mbox{if $v_3(k)=0$, and $k_3\equiv 5 \mbox{ or } 7$ {\rm{(mod 9)}}; or }\\
{} & \mbox{\quad $v_3(k) \in \{1, 4\}$, and $k_3 \equiv 2$ {\rm{(mod 3)}}; or}\\
{} & \mbox{\quad $v_3(k) = 3$, and $k_3 \not\equiv 2 \mbox{ or } 4$ {\rm{(mod 9)}}; or}\\
{} & \mbox{\quad $v_3(k) = 5$, and $k_3\equiv 1$ {\rm{(mod 3)}};}\\ [.1in]
\frac{1}{3}  &  \mbox{if $v_3(k) = 2$, and $k_3\equiv  1$ {\rm{(mod 3)}}; }\\ [.1in]

1 & \mbox{otherwise}.
\end{array}
\right.
\end{equation}
\end{proposition}

From Proposition \ref{selmeratio}, we may easily deduce explicit formulas for the number of soluble $\SL_2(\Q_p)$-orbits on $V(\Q_p)_{4k}$ (see also Corollary~\ref{solubleorbits} below).  
One may go further and describe the soluble classes as a subgroup of $(K^*/K^{*3})_{N = 1}$.   We describe this below in the cases where $p \neq 3$.

\begin{proposition}\label{solsgood}
If $k \in \Z_p$ and $E_k$ has good reduction, then the Kummer map induces an isomorphism
\[E_k(\Q_p)/\hat\phi(E_{-27k}(\Q_p)) \cong (S_0^*/S_0^{*3})_{N = 1},\]
where $S_0$ is the ring of integers of $K = \Q_p[x]/(x^2 - k)$.
\end{proposition}

\begin{proof}
See \cite[Lemma~4.1]{casVIII} or \cite[Lemma~6]{grossparson}. 

\end{proof}

\begin{proposition}\label{solsbad}
Assume $p \neq 3$ and that $E_k/\Q_p$ has bad reduction.  Then the natural map 
\[j\colon  E_k[\phi](\Q_p) \to E_k(\Q_p)/\hat\phi(E_{-27k}(\Q_p))\] 
is an isomorphism.  
\end{proposition}

\begin{proof}
We first claim that $E_k[3](\Q_p)  = E_k[\phi](\Q_p)$.  Indeed, the six other 3-torsion points on $E_k(\bar \Q_p)$ are $(\sqrt[3]{-4k}, \sqrt{-3k})$, for the six possible choices of roots in this expression.  It follows from the bad reduction of $E_k$ and Lemma \ref{reduction} that these points are defined over ramified extensions of $\Q_p$,  so they do not lie in $E_k(\Q_p)$, proving the claim.  

Next we show that $j$ is injective.  Indeed, if $P \in E_k[\phi](\Q_p)$ and $Q \in E_{-27k}(\Q_p)$ satisfy $\hat\phi(Q) = P$, then $Q \in E_{-27k}[3](\Q_p) = E_{-27k}[\hat\phi](\Q_p)$, where the equality follows from the above claim applied to $E_{-27k}$.  Thus, $P = 0$ and $j$ is injective.  

Finally, it suffices to show that $|E_k(\Q_p)/\hat\phi(E_{-27k}(\Q_p))| \leq |E_k[\phi](\Q_p)|$.  Note that the group $E_k(\Q_p)/3E_k(\Q_p)$ surjects onto $E_k(\Q_p)/\hat\phi(E_{-27k}(\Q_p))$.  Since $p \neq 3$, an argument using the formal group \cite[Lemma~12.3]{bg} shows that $E_k(\Q_p)/3E_k(\Q_p)$ has size $|E_k[3](\Q_p)|$, which is equal to $|E_k[\phi](\Q_p)|$, again by our claim above.  Therefore, we indeed have 
\[|E_k(\Q_p)/\hat\phi(E_{-27k}(\Q_p))| \leq  |E_k[3](\Q_p)| = |E_k[\phi](\Q_p)|,\] as desired.    
\end{proof}

\begin{corollary}\label{solubleorbits}
If $p \neq 3$, then the number of soluble orbits in $V(\Q_p)$ of discriminant $4k$ is equal to 
\[\begin{cases}
|E_{-27k}[\hat\phi](\Q_p)| & \mbox{if } E_k/\Q_p \mbox{ has good reduction,}\\
|E_k[\phi](\Q_p)| & \mbox{if } E_k/\Q_p  \mbox{ has bad reduction}. 
\end{cases}\]
\end{corollary}

\begin{remark}\label{3torsion}
{\em
At the start of the proof of Proposition~\ref{solsbad},
we noted that $E_k[3]$ consists of the three points of $E_k[\phi]$
together with the six points $(x,y)$ with $x^3 = -4k$ and $y^2 = -3k$.
If such a point is rational then $(k,x,y) = (-432m^6, 12m^2, 36m^3)$
for some nonzero~$m$.  In this case our elliptic curve $E_k \cong E_{-432}$
is isomorphic with the Fermat cubic $X^3 + Y^3 = Z^3$ (which is also
isomorphic with the modular curve ${\rm X}_0(27)$), and the isogenous curve
$E_{27k} \cong E_{16}$ is isomorphic with $XY(X+Y)=Z^3$.
These curves have good reduction at primes other than~$3$;
in particular, they satisfy the condition of Lemma~\ref{reduction}
for good reduction at~$2$, with minimal models
${y'}^2 + y' = {x'}^3 - 7$ and ${y'}^2 + y' = {x'}^3$ respectively.
}
\end{remark}

\subsection{Orbits over $\R$ or $\C$}\label{R}
If $F  = \R$ or $\C$, then every binary cubic form $f \in V(F)$ is reducible, so by Corollary \ref{soluble}, there is a unique $\SL_2(F)$-orbit of binary cubic forms of discriminant $4k$.  This agrees with the fact that $E_k(F)/\hat\phi(E_{-27k}(F))$ is always trivial.  We conclude that $c_\infty(\phi_k) = 1/3$ if $F = \C$, and if $F = \R$, we have:
\[c_\infty(\phi_k) = 
\begin{cases}
\frac{1}{3} & \mbox{if $k > 0$}; \\
1 & \mbox{if $k< 0$}.
\end{cases}\]

\section{$\phi$-Selmer groups and locally soluble orbits over a global field}

We let $F$ now be a global field of characteristic not 2 or 3. If $\varphi \colon A \to A'$ is an isogeny of abelian varieties over $F$, then the {\it $\varphi$-Selmer group} 
\[\Sel_\varphi(A) \subset H^1(G_F, A[\varphi])\] is the subgroup consisting of classes that are locally in the image of the Kummer map 
\[\partial_v \colon A'(F_v) \la H^1(G_{F_v}, A[\varphi])\]
for every place $v$ of $F$.  Equivalently, these are the classes locally in the kernel of the map   
\[H^1(G_{F_v}, A[\varphi]) \to H^1(G_{F_v}, A)[\varphi],\]
for every place $v$ of $F$, i.e., the classes corresponding to $\phi$-coverings having a rational point over $F_v$ for every place $v$.


Now let $V(F)^\ls \subset V(F)$ denote the set of {\it locally soluble} binary cubic forms, i.e.\ those $f(u,v) \in V(F)$ such that the equation $z^3 = f(u,v)$ has a non-zero solution over $F_v$ for every place $v$ of $F$.  Then the following result follows immediately from Corollary \ref{soluble}. 

\begin{theorem}\label{Qcor}
  Let $k \in F^*$. Then  there is a bijection between the $\SL_2(F)$-orbits on
  $V(F)^\ls$ having discriminant~$4k$ and the elements of $\Sel_{\hat\phi}(E_{-27k})$ corresponding to the isogeny $\hat\phi_k\colon E_{-27k}\to E_{k}.$  Under this bijection, the identity element of $\Sel_{\hat\phi}(E_{-27k})$ corresponds to the unique $\SL_2(F)$-orbit of reducible binary cubic forms of
  discriminant $4k$, namely the orbit of $f(x,y) = kx^3 + 3xy^2.$  Moroever, the stabilizer in $\SL_2(F)$ of any $f \in V(F)^\ls_{4k}$ is isomorphic to $E_{-27k}[\hat\phi](F)$. 
\end{theorem}

\section{Existence of integral orbits}\label{loc}

We now specialize to the case where $F = \Q$.  The main goal of this section is to prove the following theorem. 

\begin{theorem}\label{globalint}
Let $k \in 3\Z$.  Then every locally soluble orbit for the action of $\SL_2(\Q)$ on $V(\Q)_{4k}$ has an integral representative, i.e., contains an element of $V(\Z)_{4k}$.  
\end{theorem}

\begin{proof}
By Theorem \ref{hcl1} applied to the cases $D = \Z$ and $D = \Z_p$, it suffices to find a
representative in $V(\Z_p)_{4k}$ for every soluble orbit of cubic forms $f  \in V(\Q_p)_{4k}$, for each prime $p$.  Indeed, such integral representatives would correspond to valid triples over $\Z_p$, which together determine a valid triple over $\Z$.  

Let $S = \Z_p[z]/(z^2 - k)$ and $K = \Q_p[z]/(z^2 - k)$ as before, and assume for now that $p \neq 3$.  The binary cubic form $f \in V(\Q_p)_{4k}^\sol$ corresponds to an element $\delta \in (K^*/K^{*3})_{N = 1}$ in the image of the Kummer map.
Let $P = (x,y) \in E_k(\Q_p)$ be a rational point mapping to $\delta$.  By Remark~\ref{kummer}, we have $\delta = y - \tau$, where $\tau$ is the image of $z$ in $K = \Q_p[z]/(z^2 - k)$.  If $x$ and $y$ have negative valuation, then $P$ lies in the subgroup $E_{k,1}(\Q_p) \subset E_k(\Q_p)$ isomorphic to the formal group of $E_k$.  Since the formal group is pro-$p$ and $p \neq 3$, it follows that there is a point $Q \in E_{k,1}(\Q_p)$ such that $3Q = P$; hence $P$ is in $\hat\phi(E_{-27k}(\Q_p))$ and $\delta = \partial(P) = 1$.  Thus, the valid triple $(S, 1,1)$ corresponds to an integral representative in the orbit of $f$ under the bijection of Theorem \ref{hcl1}.  

We may thus assume that $x$ and $y$ are in  $\Z_p$. Define $I = \Z_p x + \Z_p \delta \subset K$.  Then 
\begin{equation*}
\tau I \;=\; \Z_p x\tau + \Z_p\delta\tau 
\;=\; \Z_p(xy-x\delta)+\Z_p(x^3-y\delta)
\;\subset\; 
I,
\end{equation*}
and therefore $I$ is an $S$-ideal.  Furthermore, we have $N(I) = x\Z_p$, and $N(\delta) = N(y-\tau) = y^2-k = x^3$.  
Finally, we note that the elements 
\begin{equation*}
\begin{array}{lcl}
\delta^{-1}(x^3) &=& y + \tau \\[.03in]
\delta^{-1}(x^2\delta) &=& x^2\\[.03in]
\delta^{-1}(x\delta^2) &=&  xy - x\tau \\[.03in]
\delta^{-1}(\delta^3) &=& y^2 + k - 2y\tau, 
\end{array}
\end{equation*}
are each contained in $S = \Z_p + \Z_p \tau$, implying that $I^3\subset \delta S$.  Thus $(I,\delta,x)$ is a valid triple for $S$ in the sense of Theorem~\ref{hcl1} and Remark~\ref{valid}, and yields the desired integral representative in $V(\Z_p)_{4k}$.

\begin{remark}{\em 
One can even {\it characterize} the integral classes in terms of solubility conditions, at least if $p \neq 3$.  Specifically, if $k \in \Z_p$ is sixth-power-free, then an $\SL_2(\Q_p)$-orbit of forms $f \in V(\Q_p)_{4k}$ contains an integral form $f_0 \in V(\Z_p)_{4k}$ if and only if $f$ is soluble over $\Q_p^\ur$, the maximal unramified extension of $\Q_p$.  }
\end{remark}

Before considering the case $p = 3$, we state a general lemma.  For any prime $p$, and for any $i \geq 0$, let $S_i$ be the order of index $p^i$ in the maximal order $S_0$ of $K = \Q_p[z]/(z^2 - k)$.

\begin{lemma}\label{inductive}
If $(I, \delta, s)$ is a valid triple for the ring $S_i$, with $i \geq 1$, then there exists an $S_{i+1}$-ideal $J$ such that $(J, \delta, s)$ is a valid triple for the ring $S_{i+1}$.  If $p = 3$, then this is true even if $i = 0$.    
\end{lemma}

\begin{proof}
This follows from the explicit bijection given in the proof of Theorem \ref{hcl1}.  If $f = [a,3b,3c,d]$ is the binary cubic form corresponding to the valid triple $(I,\delta,s)$, then $f$ has a root over $\F_p$, since 
\[\Disc(f) = -27\Delta(f) = -27\cdot \Disc(S_i)\] is divisible by $p$.  So via a change of variables we may assume that $p \mid d$.  If $\alpha \Z_p + \beta \Z_p$ is the corresponding basis of $I$, then $[p^2 a, 3pb,3c,d/p]$ is an integer-matrix binary cubic form corresponding to the triple $(p\alpha \Z_p + \beta \Z_p, \delta, s)$ for the ring $S_{i+1}$.  If $p = 3$, then $\Disc(f) = -27\Delta(f)$ is divisible by $p$, even if $i = 0$.     
\end{proof}

%

Now let $p = 3$ and let $f \in V(\Q_3)_{4k}^\sol$ correspond to $\delta \in (K^*/K^{*3})_{N = 1}$, with $K$ and $S$ as before.  If the class of $\delta$ lies in the unit subgroup $(S_0^*/S_0^{*3})_{N = 1}$, then $(S_0, \delta, s)$ is a valid triple for the ring $S_0$, and so by Lemma~\ref{inductive}, there exists $J$ such that $(J, \delta, s)$ is valid for the ring $S$.  This triple corresponds to the desired $f \in V(\Z_3)_{4k}$ in the orbit corresponding to $\delta$.   

If $\delta$ is not represented by a unit, then we must be in the case $K = \Q_3 \times \Q_3$, as this is the only \'etale $\Q_3$-algebra where there exist $\delta$'s not represented by units. Since $k \in 3\Z$, we must have $S = S_i$ for some $i \geq 1$.  We may choose $\delta$ to be of the form $(3\pi, 3u)$, for some uniformizer $\pi \in 3\Z_3$ and unit $u \in \Z_3^*$.  Then the triple $(3S_0, \delta, s)$ is valid for the ring $S_1$, so by Lemma \ref{inductive}, there exist valid triples $ (J, \delta, s)$ for $S_i$, for all $i \geq 1$, and in particular for $S$.  This gives the desired $f \in V(\Z_3)_{4k}$ in the orbit corresponding to $\delta$, and completes the proof of Theorem~\ref{globalint}.             
\end{proof}

\section{The average size of $\Sel_\phi(E_k)$}\label{proofmain}

In this section, for acceptable subsets $S \subset \Z$, we asymptotically count the number of
$\SL_2(\Q)$-classes of locally soluble integer-matrix binary cubic forms $f \in V(\Z)$ of discriminant
$3^64k$ for $|k|<X$, $k \in S$,  as $X\to\infty$.
By the work of Section~\ref{loc}, this will then allow us to deduce Theorems~\ref{main}, \ref{mainS}, and~\ref{mavg}.

\subsection{The asymptotic number of binary cubic forms
of bounded discriminant with weighted congruence conditions}

Let $V(\R)$ denote the vector space of binary cubic forms over $\R$.
Let $V^{(0)}(\R)$ denote the subset of elements in $V(\R)$ having positive
discriminant, and $V^{(1)}(\R)$ the subset of elements having negative
discriminant. We use $V^{(i)}(\Z)$ to denote $V(\Z)\cap V^{(i)}(\R)$.

Now let $T$ be any set of integral binary cubic forms that is invariant
under the action of $\SL_2(\Z)$.  Let $N(T;X)$ denote the number of
{\it irreducible} binary cubic forms contained in $T$, up to
$\SL_2(\Z)$-equivalence, having absolute discriminant at most $X$.
Then we have the following theorem counting $\SL_2(\Z)$-classes of integer-matrix 
binary cubic forms of bounded reduced discriminant, which easily follows from the work of 
Davenport~\cite{Davenport2} and Davenport--Heilbronn~\cite{dh} (see~{\cite[Theorem~19]{bv}} for this deduction):

\begin{theorem}\label{countthm} Let $T$ be any
  set of integer-matrix binary cubic forms that is defined by finitely many
  congruence conditions modulo prime powers
  and that is invariant under the action of $\SL_2(\Z)$. 
  For each prime $p$,
  let $\mu_p(T)$ denote the $p$-adic density of the $p$-adic
  closure of $T$ in $V(\Z_p)$, where $V(\Z_p)$ is equipped with the
  usual additive $p$-adic measure normalized so that
  $\mu_p(V(\Z_p))=1$.  Then
  \begin{itemize}
  \item[{\rm (a)}]
   $\displaystyle{N(T\cap V^{(0)}(\Z);X)
      = \frac{\pi^2}{12} \cdot \prod_p \mu_p(T)\cdot X + o(X);}$
   \item[{\rm (b)}]
   $\displaystyle{N(T\cap V^{(1)}(\Z);X) =
     \frac{\pi^2}{4} \cdot \prod_p \mu_p(T)\cdot X + o(X).}$
  \end{itemize}
\end{theorem}

For our particular application, we require a more general
congruence version of our counting theorem for binary cubic forms, namely,
one which not only allows appropriate infinite sets of congruence conditions
to be imposed but which also permits weighted counts of lattice points
(where weights are also assigned by congruence conditions).
More precisely, we say that a function $\phi\colon V(\Z)\to[0,1]\subset\R$ is
{\it defined by congruence conditions} if, for all primes~$p$,
there exist functions $\phi_p\colon  V(\Z_p)\to[0,1]$ such that:
\begin{itemize}
\item[(1)] For all $f\in V(\Z)$, we have $\phi(f)=\prod_p\phi_p(f)$.
\item[(2)] For each prime $p$, the function $\phi_p$ is
locally constant outside some set $Z_p \subset V(\Z_p)$ of measure zero.
\end{itemize}
Such a function $\phi$ is called {\it acceptable} if it is
$\SL_2(\Z)$-invariant and, for sufficiently large primes~$p$, we have
$\phi_{p}(f)=1$ whenever $p^2\nmid \Delta(f)$.
For such an acceptable function $\phi$, we let $N_\phi(V^{(i)}(\Z); X)$ denote
the number of $\SL_2(\Z)$-equivalence classes of elements in $V^{(i)}(\Z)$
having absolute discriminant at most~$X$, where the equivalence class of
$f\in V^{(i)}(\Z)$ is weighted by $\phi(f)$.

We then have the following generalization of Theorem~\ref{countthm}:
\begin{theorem}\label{thsquarefreebc}
  Let $\phi \colon V(\Z)\to[0,1]$ be an acceptable function that is defined by
  congruence conditions via the local functions $\phi_{p}\colon V(\Z_p)\to[0,1]$.
  Then
\begin{equation}
N_\phi(V^{(i)}(\Z);X)
  = N(V^{(i)}(\Z);X)
  \prod_{p} \int_{f\in V_{\Z_{p}}}\phi_{p}(f)\,df+o(X).
\end{equation}
\end{theorem}
\begin{proof}
The proof is exactly as in \cite[Theorem~2.21]{bs2}, though we use 
\cite[Proposition~29]{bst} in place of the uniformity estimate \cite[Theorem~2.13]{bs2}.
\end{proof}

\subsection{Weighted counts of binary cubic forms corresponding to Selmer elements}\label{wtdcount}

We wish to apply Theorem~\ref{thsquarefreebc} to the set $T$ of all integral
binary cubic forms that lie in the correspondence of Theorems~\ref{Qcor} and~\ref{globalint},
with appropriately assigned weights.
Namely, we need to count each $\SL_2(\Z)$ orbit, $\SL_2(\Z)\cdot f$,
weighted by $1/n(f)$, where $n(f)$ is equal to the number of
$\SL_2(\Z)$-orbits inside the $\SL_2(\Q)$-equivalence class of $f$
in $V(\Z)$. For this purpose, it suffices to count the number of
$\SL_2(\Z)$-orbits of locally soluble integral binary cubic forms
having bounded discriminant and no rational linear factor with each orbit
$\SL_2(\Z)\cdot f$ weighted by $1/m(f)$, where
$$m(f):=\displaystyle\sum_{f'\in O(f)}\frac{|\Aut_\Q(f')|}{|\Aut_\Z(f')|}=\displaystyle\sum_{f'\in O(f)}\frac{|\Aut_\Q(f)|}{|\Aut_\Z(f')|};$$ 
here $O(f)$ is a set of representatives for the action of $\SL_2(\Z)$
on the $\SL_2(\Q)$-equivalence class of $f$ in $V(\Z)$ and $\Aut_\Q(f)$ (resp.\
$\Aut_\Z(f)$) denotes the stabilizer of $f$ in $\SL_2(\Q)$ (resp.\
$\SL_2(\Z)$). The reason it suffices to weight by $1/m(f)$ instead of
$1/n(f)$ is that, by the proof of \cite[Lemma~22]{bst},
all but a negligible number of $\SL_2(\Z)$-orbits of
integral irreducible binary cubic forms with bounded discriminant
have trivial stabilizer in $\SL_2(\Q)$; thus all but a negligible number of 
$\SL_2(\Z)$-equivalence classes of integral binary cubic forms of bounded discriminant satisfy $m(f)=n(f)$. 
In the remainder of this subsection, we compute this weighted $p$-adic density of locally soluble
integral binary cubic forms.


For a prime $p$ and a binary cubic form $f\in V(\Z_p)$, define $m_p(f)$ by
$$m_p(f):=\displaystyle\sum_{f'\in O_p(f)}\frac{|\Aut_{\Q_p}(f')|}{|\Aut_{\Z_p}(f')|}=\displaystyle\sum_{f'\in O_p(f)}\frac{|\Aut_{\Q_p}(f)|}{|\Aut_{\Z_p}(f')|},$$
where $O_p(f)$ is a set of representatives for the action of
$\SL_2(\Z_p)$ on the $\SL_2(\Q_p)$-equivalence class of $f$ in
$V(\Z_p)$ and $\Aut_{\Q_p}(f)$ (resp.\ $\Aut_{\Z_p}(f)$) denotes the
stabilizer of $f$ in $\SL_2(\Q_p)$ (resp.\ $\SL_2(\Z_p)$). By Proposition~\ref{stabs}, we have the factorization 
\[m(f)=\prod_pm_p(f),\]
and so we have put ourselves in position to use Theorem \ref{thsquarefreebc}.

Let $S$ be an acceptable subset of $\Z$, and for each prime $p$, let $S_p$ denote the closure of
$S$ in $\Z_p$. 
Let ${B(S)}$ denote the set of all locally soluble integral binary
cubic forms having discriminant $3^6 4 k$ for $k\in S$,
and let ${B_p(S)}$ denote the
$p$-adic closure of ${B(S)}$ in $V(\Z_p)$. 

We now determine the
$p$-adic density of ${B_p(S)}$, where each element $f\in
{B_p(S)}$ is weighted by $1/m_p(f)$,
in terms of the $p$-adic integral over $S_p$ of the local Selmer ratio and 
the volume $\Vol(\SL_2(\Z_p))$ of the group $\SL_2(\Z_p)$, which is computed with respect to a fixed generator $\omega$ of the  rank 1 module of top-degree differentials of $\SL_2$ over $\Z$, so that $\omega$ is well-determined up to sign. 



\begin{proposition}\label{denel}
There is a rational number $\mathcal{J} \in \Q^\times$, independent of $S$ and $p$, such that
\[\int_{{B_p(S)}}\frac{1}{m_p(f)}df=|\J|_p\cdot \displaystyle{\Vol(\SL_2(\Z_p))\cdot \int_{k\in S}\frac{|E_{k}(\Q_p)/\hat\phi(E_{-27k}(\Q_p))|}{|E_{-27k}(\Q_p)[\hat\phi]|}dk}.\]
\end{proposition}
\begin{proof}
We make use of the facts that 
the number of $\SL_2(\Q_p)$-orbits of soluble binary cubic forms of discriminant $4k$ is equal to the cardinality of
$E_{-27k}(\Q_p)/\hat\phi(E_k(\Q_p))$ (by Theorem~\ref{Qcor}), and the cardinality of
$\Aut_{\Q_p}(f)$ is equal to the cardinality of
$E_k(\Q_p)[\hat\phi]$ (by Corollary~\ref{H1bij}); otherwise, 
the proof is identical to \cite[Proposition~3.9]{bs2}. 
\end{proof}


 
We note that the rational number $\J$ also shows up in the archimedean factor of Theorem \ref{thsquarefreebc}, since 
\begin{equation}\label{vol}
N(V^{(i)}_\Z; X) = \frac{\Vol(\SL_2(\Z) \backslash \SL_2(\R))|\J|_\infty}{n_i}X + o(X),
\end{equation}    
where $n_0 = 3$ and $n_1 = 1$; see \cite[Proposition~3.11, Remark~3.14]{bs2}.  (In fact, it turns out that $\J=3/2$, although we shall not need this fact.)
\subsection{Proof of Theorems \ref{main} and \ref{mainS}}\label{final}

\begin{theorem}\label{cong}
Let $S \subset \Z$ be any acceptable subset of integers. 
Then when all elliptic curves
 $E_k\colon y^2=x^3+k$, $k\in S$, are ordered by the absolute value of $k$,
the average size of the $\hat\phi$-Selmer group associated to
 the $3$-isogeny $\hat\phi \colon E_{-27k}\to E_{k}$ is 
\begin{eqnarray}\label{fineq}
1+ \prod_{p \leq \infty}
 \displaystyle\frac{\displaystyle\int_{k\in S_p} c_p(\hat\phi_k)dk}
  {\displaystyle\int_{k\in S_p}dk},
\end{eqnarray} 
where $S_p$ denotes the $p$-adic closure of $S$ in $\Z_p$ for $p < \infty$, and $S_\infty$ is the smallest interval in $\R$ containing $S$. 
\end{theorem}

\begin{proof}
Let $X = 3^64Y$, and $G = \SL_2$ and define $E'_k = E_{-27k}$.
By the discussion in \S\ref{wtdcount}, to count the total number of non-identity elements in the groups $\Sel_{\hat\phi_k}(E_{-27k})$ with $k < Y$, it suffices to count the number of locally soluble irreducible $\SL_2(\Z)$-orbits of binary cubic forms $f \in V(\Z)$ having discriminant in the set $3^6 4S$ and also bounded by $X$, and where each orbit is weighed by $1/m(f)$.  This weighting function is acceptable by Proposition \ref{intorb}.  Thus, by Theorem~\ref{thsquarefreebc}, Proposition~\ref{denel},  \S \ref{R}, and (\ref{vol}), the number of such weighted orbits divided by the total number of $k \in S$ with $k < Y$ approaches
\begin{equation*}\label{fineq3}
  \Vol(G(\Z)\backslash G(\R)) 
  |\J|_\infty\frac{\displaystyle{\sum_{{k\in S}\atop{|k|<Y}}}
   \frac{1}{|E'_k[\hat\phi](\R)|}}
   { \displaystyle{\sum_{{k\in S}\atop{|k|<Y}}} 1}
   \prod_p|\J|_p\Vol(G(\Z_p)) \displaystyle
   \frac{\displaystyle\prod_p\displaystyle\int_{k\in S_p}
   \frac{|E_{k}(\Q_p)/\hat\phi(E'_k(\Q_p))|}{|E'_k[\hat\phi](\Q_p)|}dk}
   {\displaystyle\prod_p\displaystyle\int_{k\in S_p}dk}
\end{equation*}
\begin{align*}
&=\frac{\displaystyle{\sum_{{k\in S}\atop{|k|<Y}}}
   \frac{1}{|E'_k[\hat\phi](\R)|}}
   { \displaystyle{\sum_{{k\in S}\atop{|k|<Y}}} 1}\prod_p \displaystyle
   \frac{\displaystyle\int_{k\in S_p}
   \frac{|E_{k}(\Q_p)/\hat\phi(E'_k(\Q_p))|}{|E'_k[\hat\phi](\Q_p)|}dk}
   {\displaystyle\int_{k\in S_p}dk}\\
   &= \prod_{p \leq \infty}
 \displaystyle\frac{\displaystyle\int_{k\in S_p} c_p(\hat\phi_k)dk}
  {\displaystyle\int_{k\in S_p}dk}
  \end{align*}
as $Y\to\infty$.  The first equality is due to the product formula 
$\prod_{ p\leq \infty} |\J|_p = 1$, and the fact that the Tamagawa number $\Vol(G(\Z)\backslash G(\R))\prod_p\Vol(G({\Z_p}))$ of $G$ is equal to $1$.  This proves   
 formula (\ref{fineq}) of Theorem~\ref{cong}, after taking into account the identity element of each $\hat\phi$-Selmer group.
\end{proof}

\begin{proof}[\bf Proof of Theorem \ref{mainS}]
The formula comes from the change of variables $k \mapsto -27k$ in the result of Theorem \ref{cong}, and the fact that $\hat\phi_k = \phi_{-27k}$.  
\end{proof}

\begin{proof}[\bf Proof of Theorem \ref{main}]  We use Proposition \ref{selmeratio} to compute the $p$-adic integrals in Theorem \ref{mainS} for the set $S = \Z$.  The result is that the product of the local densities at finite primes is equal to 
\begin{eqnarray*}
r&=&\;
\left[
 \frac{
   (1-2^{-1})
   (\frac43 + 2^{-1} + \frac43 2^{-2} + 2^{-3} +  2^{-4}+ 2^{-5})
   }{1-2^{-6}}
\right]
\label{r2} \\
 {} & &\cdot
\left[
 \frac{
  (1-3^{-1})
  (\frac53 + 2 \cdot3^{-1} + \frac23 \cdot3^{-2} + \frac73 \cdot 3^{-3}
    + 2\cdot3^{-4} + 6\cdot  3^{-5})
    }{1-3^{-6}}
\right]
\label{r3}\\
 {} & & \cdot
\prod_{p\,\equiv\, 5 \rm{\:(mod\: 6)}}
\left[\frac{
   (1-p^{-1})
   (1 + p^{-1} + \frac53 p^{-2} + p^{-3} + \frac53 p^{-4} + p^{-5})
   }{1-p^{-6}} 
\right].
\label{rp}
\end{eqnarray*}
In each Euler factor, the six-term sum records the weighted average value of $c_p(\phi)$ on $p^i \Z_p - p^{i+1}\Z_p$ for $i = 0, \ldots, 5$.  Since $E_{km^6} \cong E_k$, these averages cyclically repeat themselves for $i \geq 6$, which explains the factor $1 - p^{-6}$ in the denominator.     
\end{proof}

\subsection{Proof of Theorem \ref{mavg}}\label{Tmproof} 
Claim $(i)$ is clear, and $(ii)$ follows immediately from Proposition \ref{selmeratio}.  Claim $(iii)$ follows from Cassels' formula \cite{casVIII},
\begin{equation}\label{casformula}
c(\phi) = \dfrac{|E_{-27k}(\Q)[\hat\phi]|\cdot |\Sel_\phi(E_k)|}{|E_k(\Q)[\phi]|\cdot|\Sel_{\hat\phi}(E_{-27k})|}.
\end{equation}

To prove $(iv)$, we use Proposition \ref{selmeratio} to express $T_m$ as a disjoint union 
\[T_m = \bigcup_{n \in \Z} T_{m,n},\]
where $T_{m,n}$ is the set of all $k \in T_m$ for which there are exactly $m + n$ primes $p$ (possibly including $p = \infty$) satisfying $c_p(\phi_k) \geq 3$, with the caveat that if $c_3(\phi_k) = 9$ then $p = 3$ is counted twice.  Each $T_{m,n}$ can itself be expressed as a disjoint union
\[T_{m,n}  = \bigcup_{(\Pi_1, \Pi_2)} T_{\Pi_1,\Pi_2},\]
where the union is over all pairs $(\Pi_1, \Pi_2)$, where $\Pi_1$ is a multiset of $m+n$ primes, with each prime having multiplicity at most 1, except for $p = 3$ which may have multiplicity at most 2, and $\Pi_2$ is a set of $n$ primes, disjoint from $\Pi_1$.  The set $T_{\Pi_1, \Pi_2}$ consists of those integers $k \in T_m$ such that $c_p(\phi_k) \geq 3$ if and only if $p \in \Pi_1$ (and having multiplicity 2 if and only if $p = 3$ and $c_3(\phi_k) = 9)$, and such that $c_p(\phi_k) = 1/3$ if and only if $p \in \Pi_2$.  

Note that each set $T_{\Pi_1,\Pi_2} \subset \Z$ is acceptable, by the explicit congruence conditions of Proposition~\ref{selmeratio}.  Moreover, for each prime $p$, the function $c_p(\phi_k)$ is constant on $T_{\Pi_1,\Pi_2}$.  Thus by Theorem~\ref{mainS}, for any acceptable set $S \subset T_{\Pi_1,\Pi_2}$, the average size of $\Sel_\phi(E_k)$ for $k \in S$ is 
\[1 + \prod_p c_p(\phi_k) = 1 + c(\phi_k) = 1 + 3^m.\]
Since $c(\hat\phi) = c(\phi)^{-1}$, a similar argument shows that the average size of $\Sel_{\hat\phi}(E_{-27k})$ for $k \in S$ is $1 + 3^{-m}$. Note that any acceptable set $S$ contained in $T_m$ is necessarily contained in some $T_{\Pi_1,\Pi_2}$, so this proves Theorem \ref{mavg}(iv) in the case that $S\subset T_m$ is an acceptable set.  

If $S = T_m$, then since 
\[T_m = \bigcup_n \bigcup_{(\Pi_1,\Pi_2)} T_{\Pi_1,\Pi_2},\]
we may write $T_m$ as an ascending union $T_m = \bigcup_{i \geq 1} S_i$, where each $S_i$ is the union of all $T_{\Pi_1,\Pi_2}$ such that each prime $p \in \Pi_1 \cup \Pi_2$ is less than the $i$-th prime number $p_i$.  Thus, each $S_i$ is a finite union of acceptable sets,  and so the average size of $\Sel_\phi(E_k)$ for $k \in S_i$ is $1 + 3^m$.  Moreover, by Proposition~\ref{selmeratio}, the binary cubic forms $f \in V(\Z)$ having discriminant in the complement $T_m - S_i$ have discriminant divisible by $p_j^2$ for some $j\geq i$.  By the uniformity estimate in~\cite[Proposition~29]{bst}, we have $\sum_{k\in T_m\setminus S; |k|<X}\Sel_{\phi}(E_k)$ is $\sum_{j\geq i}O(X/p_j^2)=O(X/p_i)$, where the implied constant is independent of~$i$.  That is, the total number of Selmer elements of $E_k$ over all $k \in T_m \setminus S_i$, $|k|<X$, is $O(X/p_i)$, while the total number of Selmer elements of $E_k$ over all $|k|<X$ is of course $\gg X$.  Letting $i$ tend to infinity, we conclude that the average size of $\Sel_\phi(E_k)$ for $k \in T_m$ is also equal to $ 1+ 3^m$.  The same argument shows that the average size of $\Sel_{\hat\phi}(E_k)$ for $k \in T_m$ is  again equal to $ 1+ 3^{-m}$.   This completes the proof of Theorem~\ref{mavg}. 

\begin{remark}{\em 
If $S \subset \Z$ is any acceptable set, then the average size of $\Sel_\phi(E_k)$ for $k \in S$ is equal to $\sum_{m \in \Z} \mu(S \cap T_m)(1 + 3^m)$ and the average size of $\Sel_{\hat\phi}(E_{-27k})$ is equal to $\sum_{m \in \Z} \mu(S \cap T_m)(1 + 3^{-m})$.  Indeed, even though the sets $S \cap T_m$ are not in general acceptable, the sets of the form $S \cap T_{\Pi_1,\Pi_2}$ are, so we can argue as above. }
\end{remark}

\section{Ranks of elliptic curves}\label{ranks}
In the previous section, we computed the average size of the $\phi$-Selmer and $\hat\phi$-Selmer groups in any acceptable family of elliptic curves with $j$-invariant 0.  In this section, we deduce an upper bound on the average rank of $E(\Q)$ in such families.   We also deduce a lower bound on the density of curves~$E_k$ having rank 0 or (3-Selmer) rank 1, respectively. 

\subsection{An upper bound on the average rank of elliptic curves with vanishing $j$-invariant}
For any elliptic curve $E/\Q$ with $j$-invariant 0, we write $r(E)$, $r_3(E)$, $r_\phi(E)$, and $r_{\hat\phi}(E)$ for the ranks of the groups $E(\Q)$, $\Sel_3(E)$, $\Sel_\phi(E)$, and $\Sel_{\hat\phi}(E')$ respectively, where $E' = E/\ker \phi$.  

\begin{proposition}\label{rankineq}
For $\phi \colon E \to E'$ as above, we have
\begin{enumerate}
\item $r(E) \leq r_3(E) \leq r_\phi(E) + r_{\hat\phi}(E)$. 
\item If $c(\phi) = 3^m$, then $m \equiv r_3(E) - \dim_{\F_3}E[3](\Q) \pmod 2$.
\end{enumerate}
\end{proposition}
\begin{proof}
Claim $(i)$ follows from the exact sequence~\cite[Corollary~1]{remke},
\begin{equation}\label{remke}
0 \to \dfrac{E'(\Q)[\hat \phi]}{\phi(E(\Q)[3])} \to \Sel_\phi(E) \to \Sel_3(E) \to \Sel_{\hat\phi}(E') \to \frac{\Sha(E')[\hat\phi]}{\phi(\Sha(E)[3])} \to 0.
\end{equation}      
As $\phi$ is adjoint with respect to the Cassels--Tate pairing \cite[Theorem~1.2]{casVIII}, there is an induced non-degenerate alternating pairing on the $\F_3$-vector space $\frac{\Sha(E')[\hat\phi]}{\phi(\Sha(E)[3])}$.  The latter therefore has even $\F_3$-dimension, and $(ii)$ now follows from  (\ref{casformula}) and (\ref{remke}). 
\end{proof}

\begin{theorem}\label{mrankthm}
Let $m \in \Z$, and suppose $S \subset T_m$ is an acceptable set.  Then the $($limsup of the$)$ average rank of $E_k$, for $k \in S$ ordered by absolute value, is at most $|m| + 3^{-|m|}$.
\end{theorem}

\begin{proof} First assume that $m \geq 0$.  Then one has the general inequality
\begin{equation}\label{minequality}
2r_{\phi}(E) + 1 - 2m\leq 3^{r_\phi(E) - m},
\end{equation}
so that 
\[r_\phi(E) \leq m - \frac12 +\frac{|\Sel_\phi(E)|}{2 \cdot 3^m}.\]
By Theorem \ref{mavg}, the average size of $\Sel_{\phi}(E_k)$ for $k \in S \subset T_m$ is $1 + 3^m$, so we conclude that the limsup of the average of $r_\phi(E_k)$ is at most $m + \frac12 3^{-m}$.  Similarly, the average size of $\Sel_{\hat\phi}(E_{-27k})$ is $1 + 3^{-m}$, so we conclude from the case $m = 0$ of inequality  (\ref{minequality}) that the limsup of the average of $r_{\hat\phi}(E_k)$ is at most $\frac12 3^{-m}$.  Combining these bounds with Proposition~\ref{rankineq}(i), the limsup of the average rank of $E_k$ for $k \in S$ is seen to be at most $m + 3^{-m}$.  For $m < 0$, the roles of $\phi$ and $\hat\phi$ are reversed, so the average rank bounds that we obtain for $m$ and $-m$ are the same.      
\end{proof}

\begin{proof}[\bf Proof of Theorem \ref{rankthm}]
We observe that, by the explicit congruence description of $T_m$ in \S\ref{Tmproof}, any $k\in T_m$ (for $|m|>1$) must satisfy $p_j^2\mid k$ for some $j\geq |m|-1$, where we again use $p_i$ to denote the $i$-th prime.  Hence, by the uniformity estimate~\cite[Proposition~29]{bst}, we have that $$\sum_{|m|>M}\sum_{{\scriptstyle k\in T_m}\atop {\scriptstyle k<X}}\Sel_{\phi}(E_k)=\sum_{j\geq M}O(X/p_j^2)=O(X/p_M),$$ where the implied constant is independent of~$M$.  
Theorem \ref{mrankthm} therefore gives the following bound on the (limsup of the) average rank of the elliptic curves $E_k$ for non-zero $k \in \Z$:
\[\limsup \,\avg_{k \in \Z} r(E_k) \leq \lim_{M\to\infty} \sum_{m = -M}^M \mu(T_m)(|m| + 3^{-|m|})+O(1/p_M) =  \sum_{m = -\infty}^\infty \mu(T_m)(|m| + 3^{-|m|}),\]
where $\mu(T_m)$ denotes the density of integers in the set $T_m$.  Using Proposition \ref{selmeratio} and the explicit description given in \S\ref{Tmproof} of $T_m$ as the disjoint union 
\[T_m = \bigcup_n \bigcup T_{\Pi_1,\Pi_2},\]
 we can estimate the densities $\mu(T_m)$ to arbitrary precision.  One may also prove, e.g., that
\[\sum_{|m| > 5} \mu(T_m)(|m| + 3^{-|m|}) < 0.001.\]  
A computation in Mathematica produces the table of densities in the introduction, and it follows that the average rank of $E_k$ is less than 1.29.  We can obtain a better bound on the average rank if we use the 
the work of \cite{ruth}.  Ruth's main result---that the average size of the 2-Selmer group of $E_k$, $k\in S$, is~3---is stated for the set $S = \Z$, but the proof appears to hold for any acceptable  set~$S$.   Since the 3-Selmer---and therefore (due to the work of \cite{dokchitser}) the 2-Selmer parity---is constant for 100\% of $E_k$ for $k\in T_m$ by Theorems~\ref{mavg}(iii) and \ref{rankineq}(ii), this implies that the average rank of $E_k$, $k\in S$, is less than 4/3 on any $T_m$, giving the bound:
\[\avg_{k \in \Z} r(E) \leq \sum_{m = -1}^1 \mu(T_m)(|m| + 3^{-|m|}) + \frac43(1 - \mu(S_1 \cup S_0 \cup S_1)) < 1.21.\]
\end{proof}

%

\subsection{A lower bound on the proportion of elliptic curves with vanishing $j$-invariant having $3$-Selmer rank 0 or 1}\label{proportions}

\begin{proof}[\bf Proof of Theorem \ref{posrank0}]
Let $s_0$ denote the liminf of the natural density of $k\in T_0$ for which $r_\phi(E_k)=0$.  Then by Theorem~\ref{mavg}(iv), $s_0$ is also the liminf of the natural density of $k\in T_0$ for which $r_{\hat\phi}(E_{k})=0$.  Since the average size of $|\Sel_\phi(E_k)|$, $k\in\Z$, is 2, we must have 
$$s_0 \cdot 1 + (1-s_0) \cdot 3 \leq 2,$$
implying $s_0\geq 1/2$. It follows that a lower density of at least 1/2 of the curves $E_k$, $k\in T_0$, satisfy $r_\phi(E_k)=r_{\hat\phi}(E_{k})=r_3(E_k)=0$, and thus $r(E_k)=0$. Theorem~\ref{posrank0} follows, since the density of $T_0$ is at least $.399$, and $(1/2)(.399)> .199$. 
\end{proof}

\begin{proof}[\bf Proof of Theorem \ref{rank1}]
Let $s_1$ denote the liminf of the natural density of $k\in T_1$ for which $r_\phi(E_k)=1$.  Then by Theorem~\ref{mavg}(iv), $s_1$ is also the liminf of the natural density of $k\in T_1$ for which $r_{\hat\phi}(E_k)=0$.  Since the average size of $|\Sel_\phi(E_k)|$, $k\in T_1$, is 4, we see that
$$s_1 \cdot 3 + (1-s_1)\cdot  9 \leq 4,$$
implying $s_1\geq 5/6$. 
In conjunction with the exact sequence (\ref{remke}), this implies that a lower density of at least 5/6 of the curves $E_k$, $k\in T_1$, satisfy $r_{\hat\phi}(E_k)=0$ and $r_\phi(E_k)=r_3(E_k)=1$ (and under the assumption that $\Sha(E_k)[3^\infty]$ is finite, that $r(E_k)=1$ as well). The identical argument with $T_{-1}$ in place of $T_1$ shows that a lower density of at least 5/6 of the curves $E_k$, $k\in T_{-1}$, satisfy $r_\phi(E_k)=0$ and $r_{\hat\phi}(E_k)= 1$.  By (\ref{remke}) and the fact that $\frac{\Sha(E')[\hat\phi]}{\phi(\Sha(E)[3])}$ has even rank, these curves also satisfy $r_3(E_k)=1$ (and again $r(E_k)=1$ under the assumption that $\Sha(E_{-27k})[3]$ is trivial).   Theorem~\ref{rank1} now follows, because the density of $T_1\cup T_{-1}$ is at least $.494$, and $(5/6)(.494)> .411$.
\end{proof}

Since $.199+.411=.61$, we conclude that at least $61\%$ of all $E_k$ have rank 0 or 1, which is Corollary~\ref{rank0or1}.

\section*{Acknowledgments}

We are grateful to Benedict Gross and Ila Varma for helpful conversations.  The first author was supported by a Simons Investigator Grant and NSF grant DMS-1001828.  The second author was supported by NSF grants DMS-1100511 and DMS-1502161. The third author was partially supported by NSF grant DMS-0943832.


\begin{thebibliography}{12}

\bibitem{hcl1}
M.\ Bhargava, Higher composition laws. I. A new view on Gauss composition, and quadratic generalizations, {\it Ann.\ of Math.}\  {\bf 159} (2004), no. 1, 217--250.


\bibitem{mlocal}
M.\ Bhargava, Mass formulae for local fields, and conjectures on the density of number fields discriminants, {\it Internat.\ Math.\ Research Notices} (2007), Article ID rnm052, 20 pages.  

\bibitem{bg}
M.\ Bhargava and B.\ Gross, The average size of the $2$-Selmer group of Jacobians of hyperelliptic curves having a rational Weierstrass point. {\it Automorphic representations and $L$-functions, $23$--$91$, Tata Inst.\ Fundam.\ Res.\ Stud.\ Math.}\, {\bf 22}, Tata Inst.\ Fund.\ Res., Mumbai, 2013.

\bibitem{BEKS?}
M.\ Bhargava, N.\ D.\ Elkies, D.\ Kane, and A.\ Shnidman, forthcoming.

\bibitem{bkls}
M.\ Bhargava, Z.\ Klagsbrun, R.\ Lemke Oliver, and A.\ Shnidman, Average sizes of Selmer groups in families of quadratic twists with a 3-isogeny, preprint.  

\bibitem{bs2}
M.\ Bhargava and A.\ Shankar, Binary quartic forms having bounded invariants, and the boundedness of the average rank of elliptic curves, {\it Ann.\ of Math.} (2) {\bf 181} (2015), no. 1, 191--242.

\bibitem{bs3}
M.\ Bhargava and A.\ Shankar, Ternary cubic forms having bounded invariants, and the existence of a positive proportion of elliptic curves having rank 0, {\it Ann.\ of Math.} (2) {\bf 181} (2015), no.\ 2, 587--621.

\bibitem{bs4}
M.\ Bhargava and A.\ Shankar, The average number of elements in the 4-Selmer groups of elliptic curves is 7, preprint.

\bibitem{bs5}
M.\ Bhargava and A.\ Shankar, The average size of the 5-Selmer group of elliptic curves is 6, and the average rank is less than 1, preprint.  

\bibitem{bst}
M.\ Bhargava, A.\ Shankar, and J.\ Tsimerman.  On the Davenport-Heilbronn theorems and second order terms. {\it Invent.\ Math.} {\bf 193} (2013), no.\ 2, 439--499.


\bibitem{bv}
M.\ Bhargava and I.\ Varma, The mean number of 3-torsion elements in the class groups and ideal groups of quadratic orders, {\it Proc.\ Lond.\ Math.\ Soc.} (3) {\bf 112} (2016), no.\ 2, 235--266.

\bibitem{cas}
J.\ W.\ S.\ Cassels, {\it Lectures on elliptic curves}, London Mathematical Society Student Texts {\bf 24}, Cambridge University Press, Cambridge, 1991.

\bibitem{cas2}
J.\ W.\ S.\ Cassels, Diophantine equations with special reference to elliptic curves, {\it J.\ London Math.\ Soc.} {\bf 41} 1966, 193--291.

\bibitem{casVIII}
J.\ W.\ S.\ Cassels, Arithmetic on curves of genus 1, VIII: On conjectures of Birch and Swinnerton-Dyer, {\it J.\ Reine Angew.\ Math.} {\bf 217} (1965), 180--199.

\bibitem{Davenport2}
H.\ Davenport, On the class-number of binary cubic forms I and II,
{\it J.\ London Math.\ Soc.} {\bf 26} (1951), 183--198.


\bibitem{dh} 
H.\ Davenport and H.\ Heilbronn, On the density of discriminants of
cubic fields II, {\it Proc.\ Roy.\ Soc.\ London Ser.\ A} {\bf 322} (1971), 
no.~1551, 405--420.


\bibitem{dokchitser}
T.\ Dokchitser and V.\ Dokchitser, On the Birch--Swinnerton-Dyer quotient modulo squares, {\it Annals of Math.}\ {\bf 172} (2010), no.1, 567--596.


\bibitem{fisher}
V.\ Dokchitser, Root numbers of non-abelian twists of elliptic curves (with an appendix by Tom Fisher), {\it Proc.\ London Math.\ Soc.} (3) {\bf 91} (2005), no.\ 2, 300--324.

\bibitem{eisenstein}
G.\ Eisenstein, Th\'{e}or\`{e}mes sur les formes cubiques et solution d\'une \'{e}quation du quatrième degr\'{e} \`{a} quatre ind\'{e}termin\'{e}es, {\it J.\ Reine Angew.\ Math.} {\bf 27} (1844), 75--79.



\bibitem{fouvry}
\'E.\ Fouvry, Sur le comportement en moyenne du rang des courbes $y^2 = x^3 + k$, {\it S\'eminaire de Th\'eorie des Nombres, Paris, 1990-91, Progr.\ Math.}\ {\bf 108}, Birkh\"auser Boston, MA, 1993.


\bibitem{grossparson}
B.\ Gross, J.\ Parson, On the local divisibility of Heegner points. {\it Number theory, analysis and geometry}, 215--241, Springer, New York, 2012. 


\bibitem{HB} D.\ R.\ Heath-Brown,
 The size of Selmer groups for the congruent number problem,
 {\it Invent.\ Math.}\ {\bf 111} (1993), no.~1, 171--195.

\bibitem{Hilbert} D.\ Hilbert,
  {\em Theory of Algebraic Invariants} (1897;
   trans.\ R.~C.\ Laubenbacher, ed.~B.~Sturmfels),
  Cambridge University Press 1993.


\bibitem{Liv} E.\ Liverance,
  Binary cubic forms with many integral points, \\
{\tt http://www.kurims.kyoto-u.ac.jp/\~{}kyodo/kokyuroku/contents/pdf/0998-8.pdf}


\bibitem{Kane}
D.\ Kane,
  On the Ranks of the 2-Selmer Groups of Twists of a Given Elliptic Curve,
  preprint ({\tt arXiv:1111.2321}).

\bibitem{KT}
D.\ Kane, J.\ Thorne, On the $\phi$-Selmer groups of the elliptic curves $y^2 = x^3 - Dx$, {\it Mathematical Proceedings of the Cambridge Philosophical Society} (2016), \\ {\tt http://dx.doi.org/10.17863/CAM.1298}.

\bibitem{KMR}
Z.\ Klagsbrun, B.\ Mazur, and K.\ Rubin,
 Selmer ranks of quadratic twists of elliptic curves,
 preprint ({\tt arXiv:1111.2321v1}).

\bibitem{remke}
R.\ Kloosterman and E.\ Schaefer, Selmer groups of elliptic curves that can be arbitrarily large, {\it J.\ Number Theory} {\bf 99} (2003), no.\ 1, 148--163. 


\bibitem{PR}
B.\ Poonen and E.\ Rains,
  Random maximal isotropic subspaces and Selmer groups,
  {\em J.~Amer.\ Math.\ Soc.}, {\bf 25} (2012), no.~1, 245--269.
  arXiv: 1009.0287



\bibitem{ruth}
S.\ Ruth, A bound on the average rank of $j$-invariant zero elliptic curves, Princeton Ph.D.\ Thesis (2014).   

\bibitem{satge}
P.\ Satg\'e, Groupes de Selmer et corps cubiques, {\it J.\ Number Theory} {\bf 23} (1986), no.\ 3, 294--317.

\bibitem{schaefer}
E.\ Schaefer, Class groups and Selmer groups, {\it J.\ Number Theory} {\bf 56} (1996), no.~1, 79--114.

\bibitem{Schur} I.\ Schur,
  {\em Vorlesungen \"{u}ber Invariantentheorie},
  Grundlehren der mathematischen Wissen\-schaften {\bf 143},
  Springer: Berlin 1968.


\bibitem{selmer}
E.\ Selmer, The Diophantine equation $ax^3+by^3+cz^3=0$, {\it Acta Math.} {\bf 85} (1951), 203--362. 

\bibitem{SD}
H.\ P.\ F.\ Swinnerton-Dyer, 
 The effect of twisting on the 2-Selmer group,
 {\it Math.\ Proc.\ Cambridge Philos.\ Soc.}\ {\bf 145} (2008), no.~3, 513--526.

\bibitem{tate}
J.\ Tate,
  Algorithm for determining the type of a singular fiber in an elliptic pencil,
  {\it Lecture Notes in Mathematics} {\bf 476}
  (Modular Forms in One Variable IV, Antwerp 1972;
  B.~J.~Birch and W.~Kuyk, eds.), 33--52.

\bibitem{thorne1}
J.\ Thorne, Vinberg's representations and arithmetic invariant theory, {\it Algebra Number Theory} {\bf 7} (2013), no.\ 9, 2331--2368. 

\bibitem{thorne2}
J.\ Thorne, $E_6$ and the arithmetic of a family of non-hyperelliptic curves of genus 3, {\it Forum Math.\ Pi} {\bf 3} (2015), e1, 41 pp. 


\end{thebibliography}
\end{document}